\documentclass{siamltex}

\usepackage[utf8]{inputenc}
\usepackage{graphicx}
\usepackage{amsmath}
\usepackage{amssymb}
\usepackage{mathtools}
\usepackage{enumerate}
\usepackage{epsfig}
\usepackage{tikz}
\usepackage[ruled,vlined]{algorithm2e}
\usepackage{float} 

\newcommand{\M}{\mathcal{M}}

\newcommand{\R}{\mathbb{R}}
\newcommand{\C}{\mathbb{C}}

\newcommand{\eps}{\varepsilon}

\newcommand{\A}{{\bf A}}

\newcommand{\bfM}{{\bf M}}
\newcommand{\bfP}{{\bf P}}

\renewcommand{\S}{{\bf S}}

\renewcommand{\Re}{\mathrm{Re}\,}

\newcommand{\wh}{\widehat}
\newcommand{\wt}{\widetilde}

\newcommand{\F}{{\bf F}}

\newcommand{\U}{{\bf U}}
\newcommand{\V}{{\bf V}}

\newcommand{\Q}{{\bf Q}}
\newcommand{\I}{{\bf I}}
\newcommand{\K}{{\bf K}}

\DeclareMathOperator{\ten}{\it Ten}
\DeclareMathOperator{\mat}{\bf{Mat}}
\DeclareMathOperator{\bigtimes}{{\hbox{\large\sf X}}}
\def\bcl{\color{black}}
\def\ecl{\color{black}}

\usepackage{accents}

\allowdisplaybreaks

\title{Rank-adaptive time integration of \\tree tensor networks}
\author{Gianluca Ceruti\footnotemark[7] \and Christian Lubich\footnotemark[1] \and Dominik Sulz\footnotemark[1]}
\date{}

\begin{document}

\maketitle

\renewcommand{\thefootnote}{\fnsymbol{footnote}}

\footnotetext[1]{Mathematisches Institut,
       Universit\"at T\"ubingen,
       Auf der Morgenstelle 10,
       D--72076 T\"ubingen,
       Germany. Email: {\tt \{lubich,sulz\}@na.uni-tuebingen.de}}
\footnotetext[7]{Institute of Mathematics, EPF Lausanne, 1015 Lausanne, Switzerland. Email: {\tt gianluca.ceruti@epfl.ch}}  
    
\begin{abstract}
A rank-adaptive integrator for the approximate solution of high-order tensor differential equations by tree tensor networks is proposed and analyzed. In a recursion from the leaves to the root, the integrator updates bases and then evolves connection tensors by a Galerkin method in the augmented subspace spanned by the new and old bases. This is followed by rank truncation within a specified error tolerance. The memory requirements are linear in the order of the tensor and linear in the maximal mode dimension.
The integrator is robust to small singular values of matricizations of the connection tensors. Up to the rank truncation error, which is controlled by the given error tolerance, the integrator preserves norm and energy for Schr\"odinger equations, and it dissipates the energy in gradient systems. Numerical experiments with a basic quantum spin system illustrate the behavior of the proposed algorithm.
\end{abstract}

{
	\pagestyle{myheadings}
	\thispagestyle{plain}
	\markboth{G.~CERUTI, CH.~LUBICH AND D.~SULZ}{TIME INTEGRATION OF TREE TENSOR NETWORKS}
}

\section{Introduction}
A tree tensor network (TTN) is a tensor in a  data-sparse hierarchical format.
Our interest here is to use TTNs for the approximate solution of evolutionary tensor differential equations of high order $d$,
\begin{equation} \label{full-eq}
	\dot{A}(t) = F(t, A(t)), \quad A(t_0) = A^0 \in \C^{n_1\times\cdots \times n_d}.
\end{equation}
In particular, though by no means exclusively, 
such problems arise in quantum dynamics, where \eqref{full-eq} can represent a quantum spin system or a spatially discretized multi- to many-body Schr\"odinger equation. 
The multilayer MCTDH method in the chemical physics literature \cite{WaT03} and matrix product states and more general TTNs in the quantum physics literature \cite{ShDV06} approximate the solution of
\eqref{full-eq} by TTNs with time-dependent bases and connection tensors. Their evolution is determined by the Dirac--Frenkel time-dependent variational principle \cite{KrS81,Lu08}, which projects the right-hand side of \eqref{full-eq} orthogonally onto the tangent space of the approximation manifold $\M$ (here, TTNs of a fixed tree rank) at the current approximation $Y(t)$,
\begin{equation} \label{proj-eq}
	\dot{Y}(t) = P_{Y(t)} F(t, Y(t)), \quad Y(t_0) = Y^0 \in \M.
\end{equation}
The numerical integration of \eqref{proj-eq} encounters substantial difficulties: first,  the abstract differential equation \eqref{proj-eq} cannot be integrated as is but instead one needs differential equations for the basis matrices and connection tensors in the TTN representation. Such differential equations for the factors were derived  from \eqref{proj-eq} in \cite{WaT03}, but the resulting differential equations become near-singular in the typical presence of small singular values of matricizations of the connection tensors. This necessitates the use of tiny stepsizes for standard integrators applied to the differential equations of the factors; cf.~\cite{KiLW16}. This problem does not appear with the projector-splitting integrator, which splits the tangent space projection $P_Y$ into an alternating sum of subprojections. This splitting leads to an efficiently implementable integrator that has been shown to be robust to small singular values; see
\cite{LuO14,KiLW16} for low-rank matrix differential equations,
\cite{Lu15,LuVW18} for the approximation of tensor differential equations by Tucker tensors of given multilinear rank,
\cite{LuOV15,HaLOVV16,KiLW16} for tensor trains / matrix product states of given ranks, and recently \cite{CeLW21,LiKR21} for general tree tensor networks of given tree rank.

In many situations, it is desirable not to fix the tree rank {\it a priori} but to choose it adaptively in every time step, because the optimal ranks required for a
given approximation accuracy may vary strongly with
time and within the tree. Moreover, rank adaptivity can indicate up to which time
the solution can be approximated with a prescribed maximal rank. 
Various elaborate rank-adaptive versions of the projector-splitting integrator for tensor trains / matrix product states have recently been developed in
\cite{DeRV21,DuC21,YaW20}, which differ in the way how subspaces are augmented.

Here we follow a different approach to a rank-adaptive TTN integrator, which is not based on the projector-splitting integrator. We extend the rank-adaptive integrator for dynamical low-rank approximation of matrix differential equations recently given in \cite{CeKL22}, which updates the left and right bases and then does a Galerkin approximation in the {\it augmented subspace spanned by the old and new bases}, followed by rank truncation with a prescribed error tolerance. This approach is conceptually and algorithmically simpler than rank adaptivity in the framework of the projector-splitting integrator, and it has been shown to retain the robustness to small singular values and to have favorable structure-preserving properties. Furthermore, it has a more parallel algorithmic structure and has no substeps with propagation backward in time, as opposed to the projector-splitting integrator. It is related to the Basis-Update \& Galerkin fixed-rank integrator of \cite{CL21}, which however uses only the new bases in the Galerkin method. The novel rank-adaptive integrator
was already extended to Tucker tensors in the same paper \cite{CeKL22} where the matrix version was presented. Here we extend it to the  more intricate situation of general TTNs, study its theoretical properties and present results of numerical experiments for a problem from quantum physics.

There are two preparatory sections:
In Section~\ref{sec:ttn} we recall TTNs in the formalism of \cite{CeLW21}, which was found useful for the formulation, implementation and analysis of numerical methods for TTNs.
In Section~\ref{sec:Tucker} we recapitulate the rank-adaptive Tucker tensor integrator of \cite{CeKL22} and give a simple extension to $r$-tuples of Tucker tensors with common basis matrices.

In Section~\ref{sec:alg} we formulate the algorithm of the rank-adaptive TTN integrator. The algorithm updates and augments the bases and augments and evolves the connection tensors by a Galerkin method in a recursion from the leaves to the root, and finally truncates the ranks adaptively in a recursion from the root to the leaves.

Section~\ref{sec:robust} briefly discusses the exactness property and robust error bound (independent of small singular values) of the new TTN integrator. We do not include a proof of these fundamental properties here, because the results are essentially the same as for the TTN integrator based on the projector-splitting integrator \cite{CeLW21} and because the proof combines the proofs of the analogous results in \cite{CeKL22,CL21} and \cite{CeLW21} in a direct way.

In Section~\ref{sec:cons} we show that, up to a multiple of the truncation tolerance,  the rank-adaptive integrator conserves the norm and energy for Schr\"odinger equations and dissipates the energy for gradient systems. 

Section~\ref{sec:num} presents numerical experiments with a basic quantum spin system, the Ising model in a transverse field \cite{S73}. An interesting observation in our experiments is that for this model, the ranks and numbers of free parameters required for a TTN on a binary hierarchical tree of minimal height turn out to be significantly smaller than those required for a matrix product state (MPS), which is a TTN on a binary tree of maximal height. This may come unexpected as the Ising model has only nearest-neighbor interactions, which are well represented by an MPS. However, MPSs appear to struggle with capturing long-range effects for this model. In any case, the rank-adaptive TTN integrator proves to be a useful tool to numerically study the influence of the tree structure in the TTN approximation of a many-body quantum system.

\bcl
In the Appendix we give a derivation and error analysis of the recursive TTN rank-trunca\-tion algorithm that we use in the rank-adaptive integrator. This is based on higher-order singular value decomposition (HOSVD) \cite{LMV00} and applies to general trees. For TTNs on binary trees, a different formulation and error analysis of an HOSVD-based rank truncation algorithm was previously given in 
\cite[Section 11.4.2]{H19}.
\ecl

\section{Recap: Tree Tensor Networks} \label{sec:ttn}
In this section we recall the tree tensor network formalism of \cite{CeLW21}, which gives a recursive construction of basis matrices and connection tensors in a concise mathematical notation.
As a preparation for the TTN integrator, we also recapitulate how operators on tree tensor networks near a given starting value are reduced to operators on tensor networks on subtrees via suitable orthogonal restrictions and prolongations.

We write tensors in italic capitals and matrices in boldface capitals throughout the paper.

\subsection{Tucker tensors}
The $i$th matricization (see, e.g., \cite{KoB09}) of a tensor $A \in \C^{n_1 \times \dots  \times n_d}$ is denoted as $\mat_i (A) \in \C^{n_i \times  n_{\lnot i}}$, where
$ n_{\lnot i} = \prod_{j \neq i}^d n_j$.
 Its $k$th row  aligns all entries of $A$ that have $k$ as the $i$th subscript, usually with reverse lexicographic ordering. The inverse operation to $\mat_i$ is called tensorization and is denoted by $\ten_i$. For a tensor $A$ and a matrix $\A_{(i)}$ of compatible dimensions, we have $A=\ten_i(\A_{(i)})$ if and only if
 $\A_{(i)}=\mat_i(A)$. 
 
 The multilinear rank $(r_1,\dots ,r_d)$ of a tensor $A$ is defined as the $d$-tuple of ranks $r_i$ of the $i$th matricization $\mat_i(A)$. 
It is known from \cite{LMV00} that $A$ has multilinear rank $(r_1,\dots ,r_d)$ if and only if it has a {\it Tucker decomposition}
\begin{align}
	A = C \bigtimes_{i=1}^d \U_i, \qquad i.e.,\quad
	a_{k_1,\dots,k_d} = \sum_{l_1=1}^{r_1} \dots  \sum_{l_d=1}^{r_d} c_{l_1,\ldots,l_d} u_{k_1,l_1}^{(1)}\dots u_{k_d,l_d}^{(d)}	,
	\label{eq:tucker_dec}
\end{align}
where each {\it basis matrix} $\U_i=(u_{k,l}^{(i)})\in\C^{n_i\times r_i}$ has orthonormal columns and the {\it core tensor} $C\in\C^{r_1\times\dots\times r_d}$ has full multilinear rank $(r_1,\dots,r_d)$.

We recall the useful formula for the matricization of Tucker tensors, see \cite{KoB09},
\begin{align}
	\mat_i \left( C \bigtimes_{j=1}^d \U_j \right) = \U_i \mat_i(C) \bigotimes_{j\neq i} \U_j^\top. \label{eq:mat_tucker}
\end{align}

\subsection{Tree tensor networks} \label{subsec:ttn}
Tree tensor networks are constructed by hierarchical Tucker decompositions. For the precise formulation, we use the notation of \cite{CeLW21} and work with the following class of trees that encode the hierarchical structure.

\begin{definition}[Ordered trees with unequal leaves]
	
	Let $\mathcal{L}$ be a given finite set, the elements of which are referred to as leaves.
	We  define the set $\mathcal{T}$ of trees  $\tau$ with the corresponding set of leaves $L(\tau)\subseteq \mathcal{L}$ recursively as follows:
	
	\begin{enumerate}[(i)]
		\item
		 {\em Leaves are trees:} 
		$ \mathcal{L} \subset \mathcal{T}$,\ \text{ and }\  $L(l) := \{l\}$ for each $l \in \mathcal{L}$.
		\item  {\em Ordered $m$-tuples of trees with different leaves are again trees:} 
		If, for some $m\ge 2$,
		$$
		\tau_1, \dots, \tau_m \in \mathcal{T} 
		\quad \text{ with }\quad
		L(\tau_i ) \cap L(\tau_j ) = \emptyset \quad \forall i \neq j,
		$$
		 then their ordered $m$-tuple is in $\mathcal{T}$:		
		$$ \tau := (\tau_1, \dots, \tau_m) \in \mathcal{T} 
		, \quad \text{ and }\quad 
		L(\tau) := \dot{\bigcup}_{i=1}^m L(\tau_i) \ . 
		$$
	\end{enumerate}		
\end{definition}
The graphical interpretation is that leaves are trees and every other tree is obtained by connecting a root to several trees; see Figure~2.1.

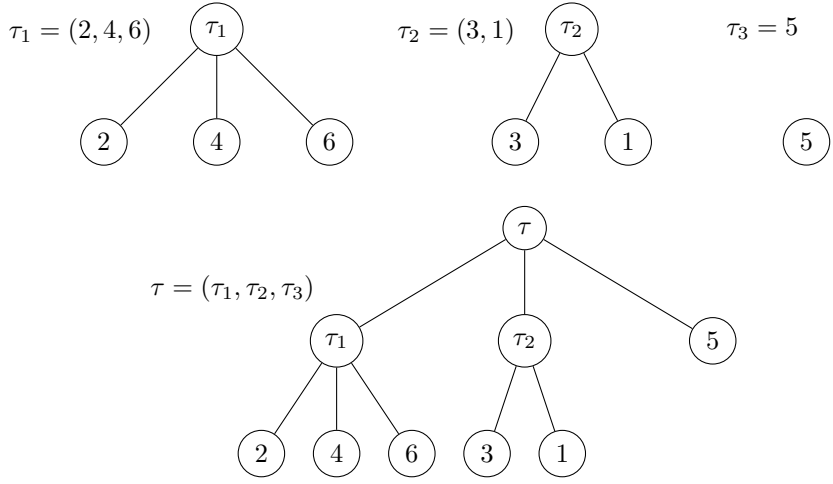
\begin{figure} 
	\label{fig:tree} 
	\begin{center}
	\begin{tikzpicture}
		\node[circle,draw]  { $\tau_1 $}
		child { node[circle,draw] {2} }
		child { node[circle,draw] {4} }
		child { node[circle,draw] {6} };
		
		\node[align=center,font=\bfseries, yshift=2em] (title) 
		at (current bounding box.west)
		{$ \tau_1 = (2,4,6) $};
	\end{tikzpicture}	
	\quad
	\begin{tikzpicture}
		\node[circle,draw] { $\tau_2$}
		child { node[circle,draw] {3} }
		child { node[circle,draw] {1} };
		
		\node[align=center,font=\bfseries, yshift=2em] (title) 
		at (current bounding box.west)
		{$ \tau_2 = (3,1) \qquad $ };
	\end{tikzpicture}
	\quad
	\begin{tikzpicture}
		\node[circle,draw] { $5$};
		 
		\node[align=center,font=\bfseries, yshift=4.3em] (title) 
		at (current bounding box.west)
		{$ \quad\tau_3 = 5 \qquad $ };
	\end{tikzpicture}
	
	\vspace{0.5cm}
	
	\begin{tikzpicture}[every node/.style={},level 1/.style={sibling distance=25mm},level 2/.style={sibling distance=10mm}]
		\node[circle,draw] { $\tau$ }
		child { node[circle,draw] { $\tau_1$ }
			child { node[circle,draw] {2} }
			child { node[circle,draw] {4} }
			child { node[circle,draw] {6} } }
		child { node[circle,draw] { $\tau_2$ }
			child { node[circle,draw] {3} }
			child { node[circle,draw] {1} } }
		child{ node[circle,draw] {5}};
		
		\node[align=center,font=\bfseries, yshift=2em] (title) 
		at (current bounding box.west)
		{$ \tau = (\tau_1, \tau_2, \tau_3) $ };
	\end{tikzpicture}
	
	  \caption{Graphical representation of a tree and three subtrees with the set of leaves $\mathcal{L}=\{1,2,3,4,5,6\}$.}
	\end{center}	
	
\end{figure} 
		The trees $\tau_1,\dots,\tau_m$ are called direct subtrees of the tree $\tau= (\tau_1, \dots, \tau_m)$, which together with direct subtrees of direct subtrees of $\tau$ etc.~are called the {\it subtrees} of~$\tau$. The subtrees are in a bijective correspondence with the vertices of the tree, by assigning to each subtree its root.
		
A partial ordering of trees is obtained
by setting for $\sigma, \tau \in \mathcal{T}$ 
\begin{align*}
	 &\sigma \leq \tau \text{ if and only if $\sigma$ is a subtree of $\tau$}, \\
	 &\sigma < \tau \text{ if and only if } \sigma \leq \tau \text{ and } \sigma \neq \tau .
\end{align*}
We fix $\bar{\tau} \in \mathcal{T}$ as a maximal tree, with $L(\bar{\tau}) = \mathcal{L}$. 
With each leaf $ l \in \mathcal{L}$ we associate a {\it basis matrix} $\U_l \in \C^{n_l \times r_l}$ of rank $r_l \leq n_l$. 
With each tree $\tau = (\tau_1 ,\dots  , \tau_m ) \leq \bar{\tau}$ we associate a {\it connection tensor} $ C_{\tau} \in \C^{r_\tau \times r_{\tau_1},\dots ,r_{\tau_m}}$ of full multilinear rank $(r_{\tau},r_{\tau_1},\dots,r_{\tau_m})$. 
A necessary condition for the connection tensor $C_\tau$ to have full multilinear rank is that (with $r_{\tau_0}:=r_\tau$)
\begin{align}
	r_{\tau_i} \leq \prod_{j \neq i} r_{\tau_j}, \qquad i=0,\dots ,m.
\end{align}
This compatibility of ranks will always be assumed in the following. \bcl We further assume $r_{\overline\tau}=1$. \ecl

From the basis matrices $\U_l$ ($l\in \mathcal{L}$) and 
connection tensors $C_\tau$ ($\tau\in\mathcal{T}\setminus\mathcal{L}$ with $\tau \leq \bar{\tau}$), a
tree tensor network is defined recursively in the following way~\cite{CeLW21}.
\begin{definition}[Tree tensor network]
	For a given tree $\bar \tau \in \mathcal{T}$ and basis matrices $\U_l$ and connection tensors $C_\tau$ as described above, we recursively define a tensor $X_{\bar \tau}$ with a tree tensor network representation (or briefly a tree tensor network) as follows: 
	\begin{enumerate}[(i)]
		\item
		For each leaf  $\,\tau = l \in \mathcal{L}$, we set 
		$$X_l := \U_l^\top \in \C^{r_l \times n_l} \ .  $$
		\item
		For each subtree $\tau = ( \tau_1, \dots, \tau_m)$  (for some $m\ge 2$) of $\bar\tau $, 
		we set \\
		$n_\tau = \prod_{i=1}^m  n_{\tau_i}$ and $\I_\tau$ the identity matrix of dimension $r_\tau$, and
		\begin{align*}
		& X_\tau := C_\tau \times_0 \I_{\tau} \bigtimes_{i=1}^m \U_{\tau_i} 
		\in \mathbb{C}^{r_\tau \times n_{\tau_1} \times \dots \times n_{\tau_m}},
		\\
		& \U_{\tau} := \mat_0( X_\tau )^\top \in \mathbb{C}^{n_\tau \times r_\tau} \ .
		\end{align*}
		The subscript $0$ in $\times_0$ and $\mat_0( X_\tau )$ refers to the mode $0$ of dimension $r_{\tau}$ in $\C^{r_\tau \times r_{\tau_1}\times\dots\times r_{\tau_m}} $. 
	\end{enumerate}
	The tree tensor network $X_{\bar{\tau}}$ (more precisely, its representation in terms of the matrices $\U_\tau$) is called {\em orthonormal} if for each subtree $\tau < \bar{\tau}$, the matrix $\U_\tau$ has orthonormal columns. 
\end{definition}

With $n=\max_\ell n_\ell$, $r=\max_\tau r_\tau$, and $m+1$ the maximal order of the connection tensors, the basis matrices and connection tensors then have less than
$dnr + d r^{m+1} \ll n^d$ entries. 

Tree tensor networks were first used in physics \cite{WaT03,ShDV06}. In the mathematical
literature, tree tensor networks with binary trees have been studied as {\it hierarchical
tensors} \cite{H19} and with general trees as {\it tensors in tree-based tensor format} \cite{FaHN15,FaHN18}; see also the survey \cite{BaSU16}. 

We define the height of a tree $\tau$ by $0$ if $\tau=l$ is a leaf, and for $\tau = (\tau_1,\dots ,\tau_m)$, we set $h(\tau) = 1 + \text{max}\{ h(\tau_1),\dots ,h(\tau_m) \}$.
Tensor trains are tree tensor networks of maximal height, Tucker tensors have height~$1$.  

The set $\mathcal{M}_{\bar\tau}=\mathcal{M}(\bar\tau, (n_l)_{l\in \mathcal{L}}, (r_\tau)_{\tau\le\bar\tau})$ of  tree tensor networks on a tree $\bar \tau\in\mathcal{T}$ with given dimensions $(n_l)_{l\in \mathcal{L}}$ and tree rank $(r_\tau)_{\tau\le\bar\tau}$ is known to be a smooth embedded manifold in the tensor space  $\C ^{\times_{l\in \mathcal{L}} n_l}$;  see \cite{UschV13}, where binary trees are considered, and also \cite{FaHN15}.

In this paper we will work with orthonormal tree tensor networks.
The following key lemma localizes orthonormality at the (small) connection tensors $C_\tau$ instead of the computationally inaccessible (huge) matrices $\U_\tau$.

\begin{lemma}  \cite[Lemma~$2.4$]{CeLW21} \label{lem:orth}
For a tree $\tau = (\tau_1, \dots  , \tau_m ) \in \mathcal{T}$, let the matrices $\U_{\tau_i}$ have orthonormal columns. Then, the matrix $\U_\tau$ has orthonormal columns if and only if the matricization $\mat_0{(C_\tau)}^\top$ has orthonormal columns. 
\end{lemma}

By the recursive definition of tree tensor networks and the lemma, we see that each tree tensor network has a representation with basis matrices $\U_\tau$ having orthonormal columns. One can obtain such an orthonormal representation by performing multiple QR-decompositions recursively from the leaves to the root. At a leaf with $l=\tau_i$ and corresponding connection tensor $C_\tau$, one computes the QR decomposition $\U_l = \Q_l\mathbf{R}_l$ and sets $\U_l = \Q_l$ and $C_\tau = C_\tau \times_i \mathbf{R}_l$.  At a connection tensor one computes the QR decomposition $ \mat_0(C_{\tau_i})^\top = \Q_{\tau_i}\mathbf{R}_{\tau_i}$ and sets $C_{\tau_i} = \ten_0(\Q_{\tau_i}^\top)$ and $C_\tau = C_\tau \times_i \mathbf{R}_{\tau_i}$. 

The contraction product of two tree tensor networks $X_\tau,Y_\tau \in \mathbb{C}^{r_\tau \times n_{\tau_1} \times \dots \times n_{\tau_m}}$  is defined by 
\begin{align*}
	\langle X_\tau,Y_\tau \rangle := \bigl( \mat_0 (X_\tau)^\top \bigr)^* \mat_0(Y_\tau)^\top \in \C^{r_\tau \times r_\tau} ,
\end{align*}
where $^*$ is the conjugate transpose.
By formula (\ref{eq:mat_tucker}), the product can be implemented recursively in an efficient way; see \cite{CeLW21} for details. This allows us to compute the product without ever computing the full tensors. 

For a maximal tree $\bar{\tau}$, we have $r_{\bar{\tau}} = 1$. Therefore the product of two tree tensor networks $\langle X_{\bar{\tau}},Y_{\bar{\tau}} \rangle \in \C$, and $\sqrt{ \langle X_{\bar{\tau}},X_{\bar{\tau}} \rangle}$ is the Euclidean norm of the orthonormal tree tensor network $X_{\bar{\tau}}$ (i.e., the Euclidean norm of the vectorized tensor).

\subsection{Reducing tree tensor network operators to subtrees}
\label{subsec:subtrees}

Suppose that on the maximal tree $\bar{\tau}$ we have a given tree tensor network $Y_{\bar\tau}^0$ and a (nonlinear) operator $F_{\bar{\tau}}$ that maps tensors in the tree tensor network representation corresponding to the tree $\bar \tau$ to tensors of the same type. We will need reduced versions $Y_{\tau}^0$ and $F_\tau$ of the given $Y_{\bar\tau}^0$ and $F_{\bar\tau}$ on all subtrees $\tau<\bar\tau$, such that $F_\tau$ maps tensors in the tree tensor network representation corresponding to the subtree $\tau$ to tensors of the same type.
In \cite{CeLW21} it is described how $Y_{\bar\tau}^0$ and $F_{\tau}$ can be recursively constructed by (linear) prolongation and restriction operators. We briefly summarize the main results and refer to \cite{CeLW21} for more details.

For a tree $\tau = (\tau_1,\dots ,\tau_m)$ we introduce the space
$
	\mathcal{V}_{\tau} = \C^{r_\tau \times n_{\tau_1} \times \dots  \times n_{\tau_m}},\ 
$
 and $\mathcal{M}_{\tau}=\mathcal{M}(\tau, (n_l)_{l\in L(\tau)}, (r_\sigma)_{\sigma\le\tau})\subset \mathcal{V}_{\tau}$ denotes the manifold of tree tensor networks of full tree rank $(r_\sigma)_{\sigma\le\tau}$, which is embedded in $\mathcal{V}_{\tau}$.
The given nonlinear operator $F_{\bar{\tau}}: \mathcal{V}_{\bar{\tau}} \rightarrow \mathcal{V}_{\bar{\tau}}$ is considered near a tree tensor network $Y_{\bar\tau}^0\in\mathcal{M}_{\bar\tau}$ (referred to as starting value). Assume by induction that $F_\tau:\mathcal{V}_{\tau} \rightarrow \mathcal{V}_{\tau}$ and a starting value $Y_{\tau}^0\in \mathcal{M}_{\tau}$ are already constructed for some subtree $\tau=(\tau_1,\dots,\tau_m)$. We then construct the nonlinear operators
$F_{\tau_i}: \mathcal{V}_{\tau_i} \rightarrow \mathcal{V}_{\tau_i}$ and the starting values $Y_{\tau_i}^0$ for the direct subtrees $\tau_i$ of $\tau$.

Consider a tree tensor network $Y_\tau^0 = C_\tau^0 \times_0 \I_\tau \bigtimes_{j=1}^m \U_{\tau_j}^0$ and define 
$$
V_{\tau_i}^0 = \mat_i ( \ten_i(\Q_{\tau_i}^{0,\top}) \times_0 \I_{\tau} \bigtimes_{j\neq i} \U_{\tau_j}^0 )^\top, 
$$ 
where $\Q_{\tau_i}^{0}$ is the unitary factor 
in the QR decomposition $ \mat_i(C_{\tau}^0)^\top = \Q_{\tau_i}^0\mathbf{R}_{\tau_i}^0$.
Then, the {\it prolongation} $\pi_{\tau,i}:\mathcal{V}_{\tau_i} \rightarrow \mathcal{V}_{\tau}$ and {\it restriction} 
$\pi_{\tau,i}^{\dagger}:\mathcal{V}_{\tau} \rightarrow \mathcal{V}_{\tau_i} $ (which both depend on the starting value $Y_\tau^0$) are defined as the linear operators
\begin{align*}
	\pi_{\tau,i}(Y_{\tau_i}) &= \ten_i ((V_{\tau_i}^0 \mat_0 (Y_{\tau_i}))^\top) \in \mathcal{V}_\tau \quad \text{for } Y_{\tau_i} \in \mathcal{V}_{\tau_i} \\
	\pi_{\tau,i}^{\dagger} (Z_\tau) &= \ten_0((\mat_i(Z_\tau)V_{\tau_i}^0)^\top) \in \mathcal{V}_{\tau_i} \quad \text{for } Z_{\tau} \in \mathcal{V}_\tau.
\end{align*}
It is shown in \cite{CeLW21} that $\pi_{\tau,i}^{\dagger}$ is both a left inverse and the adjoint of $\pi_{\tau,i}$.

For the given (nonlinear) operator $F_{\bar{\tau}}=F: \mathcal{V}_{\bar{\tau}} \rightarrow \mathcal{V}_{\bar{\tau}}$ and a tree tensor network $Y_{\bar{\tau}}^0 \in \mathcal{M}_{\bar{\tau}}$, we then define recursively for each tree $\tau = (\tau_1,\dots ,\tau_m) \leq \bar{\tau}$ and for $i=1,\dots,m$
\begin{align}
\label{F-tau-i}
	F_{\tau_i} &:= \pi_{\tau,i}^{\dagger} \circ F_\tau \circ \pi_{\tau,i} \\
\label{Y-tau-i-0}
	Y_{\tau_i}^0 &:= \pi_{\tau,i}^{\dagger}(Y_\tau^0).
\end{align}
In \cite{CeLW21} it is shown how to implement the prolongations and restrictions in an efficient recursive way. 
Moreover, the following two important properties about the prolongation and restriction are proved.
\begin{enumerate}
	\item[(i)] If the tree tensor network $Y_{\bar{\tau}}^0$ has full tree rank $(r_\sigma)_{\sigma \leq \bar{\tau}}$, then $Y_\tau^0$ has full tree rank $(r_\sigma)_{\sigma \leq \tau}$ for every subtree $\tau \leq \bar{\tau}$.
	\item[(ii)] Let $\tau = (\tau_1,\dots ,\tau_m)\in \mathcal{T}$ and $i\in\{1,\dots,m\}$. If $Y_{\tau_i} \in \mathcal{M}_{\tau_i}$, then the prolongation  $\pi_{\tau,i}(Y_{\tau_i})$ is in $\mathcal{M}_\tau$.
\end{enumerate}
\section{Rank-adaptive integrator for extended Tucker tensors} \label{sec:Tucker}
We recapitulate the rank-adaptive integrator for Tucker tensors of \cite{CeKL22} in Section \ref{subsec_rank_adapt} and we extend the algorithm to $r$-tuples of Tucker tensors with common basis matrices in Section \ref{subsec_extend_rank_adapt}. This extension will be important for formulating the rank-adaptive integrator for tree tensor networks.
\subsection{Recap: rank-adaptive integrator for Tucker tensors} \label{subsec_rank_adapt}
We start from a Tucker tensor
\begin{align*}
Y^0 = C^0 \bigtimes_{i=1}^{d} \U_i^0,
\end{align*}
where the basis matrices $\U_i^0\in \C^{n_i\times r_i^0}$ have orthonormal columns and $C^0\in \C^{r_1^0\times\dots\times r_d^0}$ is a tensor of full multilinear rank $(r_1^0,\dots,r_d^0)$. 

\bcl
Rank-adaptivity requires procedures to increase the rank as well as to truncate the rank. The truncation is usually done by a higher-order singular value decomposition (HOSVD) \cite{LMV00}. What distinguishes different approaches to rank-adaptivity is the way how the basis matrices and the core tensor are augmented;  cf. \cite{CeKL22,DeRV21,DS14,DuC21,YaW20}. This requires extra information that is not available from the approximation to the solution at a single point in time. The approach to rank augmentation taken in \cite{CeKL22} is particularly simple and turns out to be very effective and to enhance the qualitative properties of the integrator. A time step from $t_0$ to $t_1$ of the rank-adaptive {\it Basis-Update \& Galerkin} (BUG) integrator for Tucker tensors proposed in \cite{CeKL22} proceeds as follows:
\begin{enumerate}
\item 
(Updated and augmented bases) For $i=1,\dots,d$ (in parallel),  update the basis matrices $\U_i^0\in \C^{n_i\times r_i^0}$ to $\U_i'\in \C^{n_i\times r_i^0}$ and compute an orthonormal basis $\wh \U_i \in \C^{n_i\times \wh r_i}$ (with $\wh r_i \le 2 r_i^0$, typically $\wh r_i = 2 r_i^0$) of the subspace spanned by the columns of  both the {\it old and new} basis matrices $\U_i^0$ and $\U_i'$.
\item (Galerkin method in the augmented subspace $\wh {\mathcal{X}}:= \C^{\wh r_1\times\dots\times \wh r_d} \bigtimes_{i=1}^d \wh{\U}_i$)
\\
Project the tensor differential equation orthogonally onto the space $\wh {\mathcal{X}}$ and solve the initial value problem from $t_0$ to $t_1$,
using the orthogonally projected starting tensor. This yields an update of the initial core tensor $C^0\in \C^{r_1^0\times\dots\times r_d^0}$ 
to an augmented core tensor  ${\wh C}^1\in \C^{\wh r_1\times\dots\times \wh r_d}$, obtained as the solution at time $t_1$ of a differential equation for the core tensor starting from an augmented core tensor ${\wh C}^0\in \C^{\wh r_1\times\dots\times \wh r_d}$ given as
${\wh C}^0 = C^0 \bigtimes_{i=1}^d (\wh \U_i^* \U_i^0)$.
\item (Rank truncation) Truncate the updated and augmented Tucker tensor\\ $\wh Y^1 ={\wh C}^1 \bigtimes_{i=1}^d \wh \U_i$  within a given tolerance to
a modified multilinear rank $(r_1^1,\dots,r_d^1)$, using HOSVD. 
\end{enumerate}
This procedure yields the updated Tucker tensor in factorized form,
$$
Y^1 = C^1 \bigtimes_{i=1}^{d} \U_i^1,
$$
where the basis matrices $\U_i^1\in \C^{n_i\times r_i^1}$ have orthonormal columns and $C^1\in \C^{r_1^1\times\dots\times r_d^1}$ is a tensor of full multilinear rank $(r_1^1,\dots,r_d^1)$. 
This serves as the starting value for the next time step, and so on.

In 1., the $i$th basis is updated by solving the projected tensor differential equation for $Y_i(t)\in \mathcal{X}_i := C^0 \bigtimes_{j\ne i} \U_j^{0} \times_i \C^{n_i \times r_i^0}\subset \C^{n_1\times\dots\times n_d}$, where only the matrix in the $i$th mode is varied:
\begin{align*}
&\langle \dot Y_i(t),X_i \rangle = \langle F(t,Y_i(t)), X_i \rangle & \forall\, X_i \in 
\mathcal{X}_i ,
\\
& \langle Y_i(t_0),X_i \rangle = \langle Y^0, X_i \rangle   &\forall\, X_i \in 
\mathcal{X}_i ,
\end{align*}
which is equivalent to the $n_i \times r_i^0$ matrix differential equation in Algorithm~\ref{alg:Phi-i}. There, an update $\U_i'$ of the orthonormal basis would be obtained as the orthogonal factor in a QR decomposition of $\K_i(t_1)$ \cite{CL21}. Since $\K_i(t_1)$ and $\U_i'$ have the same range, we build the augmented orthonormal basis $\wh\U$ directly from $(\K_i(t_1),\U_i^0)$ instead of the combined new and old bases $(\U_i',\U_i^0)$.

In 2., the core tensor is updated by solving the projected tensor differential equation for $\wh Y(t)\in \wh{\mathcal{X}} := \C^{\wh r_1\times\dots\times \wh r_d} \bigtimes_{i=1}^d \wh\U_i\subset \C^{n_1\times\dots\times n_d}$, where only the core tensor is varied:
\pagebreak[2]
\begin{align*}
&\langle \dot {\wh Y}(t), \wh X \rangle = \langle F(t,\wh Y(t)), \wh X \rangle & \forall\, \wh X \in 
\wh{\mathcal{X}},
\\
& \langle \wh Y(t_0),\wh X \rangle = \langle Y^0, \wh X \rangle  & \forall\, \wh X \in 
\wh{\mathcal{X}},
\end{align*}
which is equivalent to the tensor differential equation in Algorithm~\ref{alg:Psi}.
\ecl

For the precise formulation of the algorithm, it is convenient to introduce subflows $\Phi^{(i)}$ and $\Psi$ which correspond to the updates of the basis matrices and the core tensor, respectively. In addition, as formulated in \cite{CeKL22}, there is the rank truncation algorithm $\Theta$ via HOSVD depending on a given truncation tolerance $\vartheta$.

The approximation $Y^1$ after one time step at time $t_1$ is then obtained by Algorithm~\ref{alg:main}, in which matrices and tensors with doubled rank carry a hat. The maps $\Phi^{(i)}$ and $\Psi$  are defined in Algorithms \ref{alg:Phi-i} and \ref{alg:Psi}, respectively. 
The algorithm can be written schematically as 
\begin{align*}
	Y^1 = \Theta \circ \Psi \circ (\Phi^{(1)},\dots,\Phi^{(d)}) (Y^0).
\end{align*}
The matrix differential equations for $\K_i(t)\in \C^{n_i\times r_i^0}$ in Algorithm 2 and the 
tensor differential equation for $\wh C(t)\in \C^{\wh r_1 \times \dots\times \wh r_d}$ 
are solved approximately using a standard integrator such as a Runge--Kutta method. 

\bigskip
\begin{algorithm}[H]
 \label{alg:main}
	\caption{Rank-adaptive Tucker integrator }
	
	\SetAlgoLined
	\KwData{ 
		Tucker tensor $Y^0 = C^0 \bigtimes_{i=1}^{d} \U_{i}^0$ in factorized form of multilinear rank $(r_{1}^0,\dots ,r_{d}^0)$, function $F(t,Y)$, $t_0,t_1$, tolerance parameter $\vartheta$
	}
	\KwResult{Tucker tensor $Y^1 = C^1 \bigtimes_{i=1}^{d} \U_{i}^1$ in factorized form of multilinear rank $({r}_{1}^1,\dots,{r}_{d}^1)$, where ${r}_{i}^1 \leq 2r_{i}^0$
	}
	\Begin{
		\For{$i=1:d$ {\rm in parallel}}{
			compute $ [ \wh\U_{i},\wh\bfM_{i} ] = \Phi^{(i)} \bigl( Y^0 ,F, t_0,t_1 \bigr)$ \\
			{\tt \% update and augment the $i$th basis matrix (see Algorithm 2)}
   	         } 
		compute $ \wh C^1 = \Psi \bigl( C^0,( \wh\U_{i} )_{i=1}^d , ( \wh\bfM_{i} )_{i=1}^d,F, t_0,t_1  \bigr)$ \\
		{\tt \% augment and update the core tensor (see Algorithm 3)}\\[2mm]
		set $\wh Y^1 = \wh C^1 \bigtimes_{i=1}^{d} \wh \U_{i} $ \\	
		compute $ Y^1 = \Theta\bigl(\wh Y^1, \vartheta \bigr) $ \quad {\tt \% rank truncation}
	}
\end{algorithm}

\bigskip
\begin{algorithm}[H]
\label{alg:Phi-i}
	\caption{Subflow $\Phi^{(i)}$ (update and augment the $i$th basis matrix))}
	
	\SetAlgoLined
	\KwData{ 
		Tucker tensor $Y^0 = C^0 \bigtimes_{j=1}^{d} \U_j^0$ of multilinear rank $(r_1,\dots,r_d)$
		in factorized form, function $F(t,Y)$, $t_0,t_1$
	}
	\KwResult{Updated and augmented basis matrix $\wh\U_i\in\C^{n_i\times \wh r_i}$ (typically $\wh r_i=2r_i$) with orthonormal columns, auxiliary matrix $\wh\bfM_i$
	}
	\Begin{
		compute the QR-decomposition $\mat_i({C^0})^\top = \Q_i^0\S_i^{0,\top}$; \\
		solve the $n_i \times r_i$ matrix differential equation from $t_0$ to $t_1$
		\begin{equation*}
			\dot{\K}_i(t) = \F_i \bigl( t, \K_i(t)\V_i^{0,\top} \bigr) \V_i^0, \quad\ \K_i(t_0) = \U_i^0\S_i^0,
		\end{equation*}
		with $\F_i(t,\cdot ) := \mat_i \circ F(t,\cdot) \circ \ten_i$ and $\V_i^{0,\top} := \Q_i^\top \bigotimes_{j \neq i}^{d} \U_j^{0,\top}$; \\
		compute $\wh \U_i$ as an orthonormal basis of the range of the $n_i \times 2{r}_i$ matrix $\left( \K_i(t_1), \U_i^0 \right)$; \\
		set $\wh\bfM_i = \wh\U_i^* \U_i^0$. \\
	}
\end{algorithm}

\bigskip
\begin{algorithm}[H]
\label{alg:Psi}
	\caption{Subflow $\Psi$ (augment and update the core tensor)}
	
	\SetAlgoLined
	\KwData{core tensor $C^0\in \C^{r_1\times\dots\times r_d}$, augmented basis matrices $\wh\U_{i}\in\C^{n_i\times \wh r_i}$ with orthonormal columns, auxiliary matrices $\wh\bfM_{i}\in \C^{\wh r_i \times r_i}$,  function $F(t,Y)$, $t_0,t_1$
	}
	\KwResult{core tensor $\wh C^1\in \C^{\wh{r}_1 \times \dots  \times \wh{r}_d}$}
		\Begin{
		solve the $\wh{r}_1 \times \dots  \times \wh{r}_d$ tensor differential equation from $t_0$ to $t_1$,
		\begin{equation*}
			\dot{\wh C}(t) = F \bigl( t,\wh C(t) \bigtimes_{i=1}^{d} \wh\U_{i} \bigr) \bigtimes_{i=1}^{d} \wh\U_{i}^*, \quad \wh C(t_0) = C^0 \bigtimes_{i=1}^{d} \wh\bfM_{i};	
		\end{equation*}
	set $\wh C^1 = \wh C(t_1)$	
	}
\end{algorithm}

\subsection{Rank-adaptive integrator for extended Tucker tensors} \label{subsec_extend_rank_adapt}
Consider a tensor in Tucker format of multilinear rank $(r_0,r_1,\dots ,r_d)$, where in the $0$-mode only an $r_0 \times r_0$ identity matrix $\I=\I_{r_0}$ appears, i.e.,
\begin{align*}
	Y^0 = C^0 \times_0 \I \,\bigtimes_{j=1}^{d} \U_j^0.
\end{align*}
This can be viewed as an $r_0$-tuple of Tucker tensors of order $d$ with common basis matrices,
\begin{align*}
	Y^0 = \bigl( C^0_1 \,\bigtimes_{j=1}^{d} \U_j^0, \dots, C^0_{r_0}  \,\bigtimes_{j=1}^{d} \U_j^0\bigr).
\end{align*}

As described in Section \ref{subsec_rank_adapt}, we obtain the approximation $Y^1$ by 
\begin{align*}
	Y^1 = \Theta \circ \Psi \circ (\Phi^{(0)}, \Phi^{(1)},\dots,\Phi^{(d)})  (Y^0).
\end{align*}
The following lemma shows that the subflow $\Phi^{(0)}$ can be ignored in the algorithm.

\begin{lemma}\label{lem:trivial}
With an appropriate choice of orthogonalization, the action of the
subflow $\Phi^{(0)}$ on $Y^0$ becomes trivial, i.e.,
\begin{align*}
	\Phi^{(0)}(Y^0) = [\wh\U_0, \wh \bfM_0] = [\I,\I].
\end{align*} 
\end{lemma}
\begin{proof}
In the subflow $\Phi^{(0)}$ we solve the $r_0 \times r_0$ matrix differential equation
\begin{align*}
	\dot{\K}_0(t) = \F_i \left( t, \K_0(t)\V_0^{0,\top} \right) \V_0^0, \ \ \ \K_0(t_0) = \U_0^0\S_0^0.
\end{align*}
We define $\K_0^1 = \K_0(t_1)$ as the solution at time $t_1$. Obviously the matrix
\begin{align*}
	( \K_0^1,\U_0^0 ) = (\K_0^1,\I) \in \C^{r_0 \times 2r_0}
\end{align*}
has exactly rank $r_0$. Therefore, the columns of $\wh\U_0 = \I$ form an orthonormal basis of the range of this matrix. Moreover, $\wh\bfM_0 = \wh\U_0^*\U_0^0 = \I^*\I = \I$. 
\end{proof}

In view of Lemma~\ref{lem:trivial}, we write a step of the extended Tucker integrator schematically as
\begin{align*}
	Y^1 = \Theta \circ \Psi \circ (\Phi^{(1)},\dots,\Phi^{(d)})  (Y^0).
\end{align*}

\section{Rank-adaptive tree tensor network integrator} \label{sec:alg}
We present a rank-adaptive integrator for orthonormal tree tensor networks, which evolves the basis matrices and connection tensors.
If $F(t,Y)$ can be evaluated for a tree tensor network $Y$ from its basis matrices and connection tensors, then the algorithm proceeds without ever computing a full tensor. The integrator relies on the extended rank-adaptive integrator of Section \ref{sec:Tucker}. It uses a recursive formulation that is notationally similar to that of the tree tensor network integrator of \cite{CeLW21}. While the latter algorithm is based on the projector-splitting Tucker tensor integrator of \cite{LuVW18}, the integrator considered here is based on the substantially different rank-adaptive Basis Update \& Galerkin Tucker tensor integrator of \cite{CeKL22}. As a consequence, the recursive TTN algorithm presented here is very different from that of \cite{CeLW21} in the algorithmic structure and details. Among other differences, we mention that it is more parallel and does not use backward time evolutions, which are problematic for strongly dissipative problems. Rank-adaptivity is built in in a very simple way, which would not be possible for the projector-splitting TTN algorithm of \cite{CeKL22}.

\subsection{Recursive rank-augmenting TTN integrator}

Consider a tree $\tau = (\tau_1, \dots, \tau_m)$ and
an extended Tucker tensor $Y_\tau^0$ associated with the tree $\tau$,
$$ Y_{\tau}^0  = C_\tau^0 \times_0 \I_\tau \bigtimes_{i=1}^m \U_{\tau_i}^0 \ . $$
Applying the extended Tucker integrator with the function $F_\tau$ without rank truncation, we obtain 
$$ 	\wh Y_\tau^1 = \Psi_\tau \circ (\Phi^{(1)}_\tau, \dots , \Phi^{(m)}_\tau) (Y_\tau^0) \ .$$
We recall that the subflow $\Phi^{(i)}_\tau$ gives the update process of the basis matrix $\U_{\tau_i}^0 \in \R^{n_{\tau_i} \times r_{\tau_i}}$. The extra subscript $\tau$ indicates that the subflow is computed for the function $F_\tau$.

We have two cases: 

\begin{enumerate}[(i)]
	\item
		If $\tau_i$ is a leaf, then we directly apply the subflow $\Phi^{(i)}_\tau$ and update and augment the basis matrix. 
	\item
		Else, we apply the algorithm approximately and recursively. The procedure will still be called $\Phi_{\tau}^{(i)}$. 
		We tensorize the basis matrix and we use new initial data $Y_{\tau_i}^0=\pi_{\tau,i}^\dagger Y_\tau^0$ and the function $F_{\tau_i}= \pi_{\tau,i}^\dagger \circ F_\tau \circ \pi_{\tau,i}$ as described in Section~\ref{subsec:subtrees}. 
		\end{enumerate}
We can now formulate the rank-adaptive integrator for tree tensor networks. It is composed of rank-augmenting integration steps
(Algorithm~\ref{alg:main-tau} with Algorithms~\ref{alg:Phi-tau-i} and \ref{alg:Psi-tau}) followed by a final rank truncation (Algorithm~\ref{alg:truncation-tau}).

\bcl In Algorithm 4, $\wh\U_{\tau_i}\in \C^{n_{\tau_i} \times \wh r_{\tau_i}}$ is the updated and augmented basis matrix (in factorized form unless $\tau_i$ is a leaf), typically with
$\wh r_{\tau_i} =2 r_{\tau_i}^0$, and $\wh \bfM_{\tau_i}= \wh\U_{\tau_i}^*\U_{\tau_i}^0 \in \C^{\wh r_{\tau_i} \times r_{\tau_i^0}}$. The tensor $\wh{C}_\tau^0 \in \C^{\wh r_\tau \times \wh r_{\tau_1}\times \dots \times \wh r_{\tau_m}}$ is the augmented initial connection tensor, and $\wh{C}_\tau^1$ of the same dimension is the updated augmented tensor. $\wh\U_{\tau_i}$ and $\wh{C}_\tau^1$  are the quantities of interest, whereas $\wh \bfM_{\tau_i}$ and $\wh{C}_\tau^0 $ are auxiliary quantities that are passed through in the recursion.
\ecl

The difference to the Tucker integrator is that now the subflow $\Phi_\tau^{(i)}$ 
uses the TTN integrator for the subtrees in a recursion from the leaves to the root. This approximate subflow is defined in close analogy to the subflow $\Phi^{(i)}$ for Tucker tensors, but the differential equation is solved only approximately by recurrence unless $\tau_i$ is a leaf.

\bigskip
\begin{algorithm}[H]
	\caption{Rank-augmenting TTN integrator }
	\label{alg:main-tau}
	
	\SetAlgoLined
	\KwData{ 
		tree $\tau = (\tau_1,\dots ,\tau_m)$, TTN $Y_{\tau}^0 = C_{\tau}^0 \times_0 \I_\tau \bigtimes_{i=1}^{m} \U_{\tau_i}^0$ in factorized form of 
		tree rank $(r_\sigma^0)_{\sigma\le\tau}$, function $F_{\tau}(t,Y_\tau)$, $t_0,t_1$
	}
	\KwResult{TTN $\wh Y_{\tau}^1 = \wh C_{\tau}^1 \times_0 \I_\tau \bigtimes_{i=1}^{d} \wh \U_{\tau_i}$ 
	in factorized form of augmented tree rank $(\wh r_\sigma)_{\sigma\le\tau}$ with $\wh r_\sigma\le 2 r_\sigma^0$,
	augmented connection tensor~$\wh{C}_\tau^0$
	}
	\Begin{
		\For{$i=1:m$ {\rm in parallel}}{
			compute $ [ \wh\U_{\tau_i},\wh \bfM_{\tau_i} ] = \Phi_{\tau}^{(i)} \bigl( Y_\tau^0,F_{\tau}, t_0,t_1 \bigr) $ \\
						{\tt \% update and augment the basis matrix for subtree $\tau_i$ (see Algorithm 5)}
		}
		compute $ [\wh C_{\tau}^1,\wh C_\tau^0] = \Psi_{\tau} \left( C_{\tau}^0,( \wh\U_{\tau_i})_{i=1}^m , ( \wh\bfM_{\tau_i})_{i=1}^m,F_{\tau}, t_0,t_1  \right)  $ \\
	{\tt \% augment and update the connection tensor (see Algorithm 6)}\\[2mm]
		set $\wh{Y}_{\tau}^1 = \wh C_{\tau}^1 \times_0 \I_\tau \bigtimes_{i=1}^{m} \wh\U_{\tau_i} $ 
	}
\end{algorithm}

\bigskip
\begin{algorithm}[H]
	\caption{Subflow $\Phi_{\tau}^{(i)}$ (update and augment a basis matrix)}
	\label{alg:Phi-tau-i}
	\SetAlgoLined
	\KwData{tree $\tau = (\tau_1,\dots ,\tau_m)$, TTN $Y_{\tau}^0 = C_{\tau}^0 \times_0 \I_\tau \bigtimes_{i=1}^{m} \U_{\tau_i}^0$ in factorized form of 
		tree rank $(r_\sigma^0)_{\sigma\le\tau}$, with $\U_{\tau_i}^0=\mat_0(X_{\tau_i}^{0})^\top\in\C^{n_{\tau_i}\times r_{\tau_i}^0}$, function $F_{\tau}(t,Y_\tau)$, $t_0,t_1$
	}
	\KwResult{ $\wh\U_{\tau_i}=\mat_0(\wh X_{\tau_i})^\top\in\C^{n_{\tau_i}\times \wh r_{\tau_i}}$ (typically $\wh r_{\tau_i}=2r_{\tau_i}^0$) in factorized form, auxiliary matrix $\wh\bfM_{\tau_i}\in\C^{\wh r_{\tau_i} \times r_{\tau_i}^0}$
	}
	\Begin{
		compute a QR-decomposition $\mat_i({C_{\tau}^0})^\top = \Q_{\tau_i}^0\S_{\tau_i}^{0,\top}$; \\
		set $Y_{\tau_i}^0 = X_{\tau_i}^0 \times_0 \S_{\tau_i}^{0,\top} $ \\
		\uIf{$\tau_i = l $ {\rm is a leaf}}{
			solve the $n_l \times r_l^0$ matrix differential equation
			\begin{align*}
				\dot{Y}_{l}(t) = F_{l}(t,Y_{l}(t)), \ \ \ Y_{l}(t_0) = Y_{l}^0\in \C^{r_l^0\times n_l}\,;
			\end{align*}
			compute $\wh \U_{l}\in \C^{n_l \times \wh r_l}$ with $\wh r_l\le 2 r_l^0$ as an orthonormal basis of the range of the $n_l \times 2{r}_l^0$ matrix  $\bigl( Y_{l}(t_1)^\top,\U_{l}^{0} \bigr) $; \\
			set $\wh \bfM_{l} = \wh\U_{l}^*\U_{l}^0\in \C^{\wh r_l\times r_l^0}$
		}
		\Else {
			$[\wh Y_{\tau_i}^{1},\wh {C}_{\tau_i}^0] = \text{\it Rank-augmenting TTN integrator\/}(\tau_i,Y_{\tau_i}^0,F_{\tau_i},t_0,t_1);$ \\
			\qquad\qquad\quad  \% {\tt recursive computation on subtrees}\\
			compute an orthonormal basis $\wh \Q_{\tau_i} $ of rank $\wh{r}_{\tau_i} \leq 2r_{\tau_i}^0$ of the range of $\bigl( \mat_0(\wh C_{\tau_i}^1)^\top , \mat_0(\wh {C}_{\tau_i}^0)^\top \bigr) $, where $\wh C_{\tau_i}^1$ is the connection tensor of $\wh Y_{\tau_i}^1$; \\
			set $\wh\U_{\tau_i} = \mat_0(\wh X_{\tau_i})^\top$, 
			where the orthonormal TTN $\wh X_{\tau_i}$ is obtained from $\wh Y_{\tau_i}^1$ by 
		replacing
		the connection tensor with $\widehat C_{\tau_i}= \ten_0(\widehat {\mathbf{Q}}_{\tau_i}^{\top})$; \\
			set $\wh\bfM_{\tau_i} = \wh\U_{\tau_i}^* \U_{\tau_i}^0 
			\in \C^{\wh r_{\tau_i} \times r_{\tau_i}^0}$ (computed as $\langle \wh X_{\tau_i}, X_{\tau_i}^0 \rangle$)
		}
	}
\end{algorithm}

%

\bigskip
\begin{algorithm}[H]
\label{alg:Psi-tau}
	\caption{Subflow $\Psi_\tau$ (augment and update the connection tensor)}
	
	\SetAlgoLined
	\KwData{tree $\tau = (\tau_1,\dots ,\tau_m)$, connection tensor $C_\tau^0\in \C^{r_\tau^0 \times r_{\tau_1}^0\times\dots\times r_{\tau_m}^0}$, augmented basis matrices $\wh\U_{\tau_i}$ in factorized form, auxiliary matrices $\wh\bfM_{\tau_i}\in \C^{\wh r_{\tau_i} \times r_{\tau_i}^0}$,  function $F_\tau(t,Y)$, $t_0,t_1$
	}
	\KwResult{connection tensors $\wh C_\tau^1,\wh C_\tau^0\in \C^{r_\tau \times \wh{r}_{\tau_1} \times \dots  \times \wh{r}_{\tau_m}}$}
		\Begin{
		set $\wh{C}_{\tau}^0 = C_\tau^0 \bigtimes_{i=1}^{m} \wh\bfM_{i}$\,; \\
		solve the $ r_\tau \times \wh{r}_{\tau_1} \times \dots  \times \wh{r}_{\tau_m}$ tensor differential equation from $t_0$ to $t_1$
		\begin{equation*}
			\dot{\wh C_\tau}(t) = F_\tau \bigl( t,\wh C_\tau(t) \bigtimes_{i=1}^{m} \wh\U_{\tau_i} \bigr) \bigtimes_{i=1}^{m} \wh\U_{\tau_i}^*, \quad \wh C_\tau(t_0) = \wh C_\tau^0\,;	
		\end{equation*}
		\\
	set  $\wh C_\tau^1 = \wh C_\tau(t_1)$	
	}
\end{algorithm}

\bigskip
%
We emphasize that the huge matrices $\wh\U_{\tau}$ (when the tree $\tau$ is not a leaf) are never computed as matrices but only their factors in the TTN representation: the connection tensors $\wh C_\sigma$ for subtrees $\sigma \le \tau$ that are not leaves, and the basis matrices $\wh \U_l$ for the leaves.

The low-dimensional matrix differential equations for $Y_{l}(t)$ and low-dimen\-sional tensor differential equations for $C_\tau(t)$ 
are solved approximately using a standard integrator, e.g., an explicit or implicit Runge--Kutta method or an exponential Krylov subspace method when $F$ is linear; note that only linear differential equations appear in the algorithm for linear $F$.

\bcl
\bigskip
\begin{algorithm}[H]
	\caption{Rank truncation $\Theta_\tau$}
	\label{alg:truncation-tau}
	
	\SetAlgoLined
	\KwData{ 
		tree $\tau=(\tau_1,\dots,\tau_m)$, TTN in factorized form\\
		$\qquad \wh X_\tau = \wh C_\tau \times_0 \I_\tau \bigtimes_{i=1}^m \wh \U_{\tau_i}$
		of tree rank $(\wh r_\sigma)_{\sigma\le\tau}$\\
		$\qquad\text{with }  \wh\U_{\tau_i} =\mat_0(\wh X_{\tau_i})^\top$ for a sub-TTN $\wh X_{\tau_i}$,\\
		$\qquad\text{tolerance parameter} \ \vartheta$
	}
	\KwResult{TTN in factorized form ${X}_\tau = C_\tau \times_0 {\I}_\tau \bigtimes_{i=1}^m {\U}_{\tau_i}$\\
	         $\qquad$of tree rank $(r_\sigma)_{\sigma\le\tau}$ with adaptively chosen $r_\sigma \le \wh r_\sigma$\\	         
		$\qquad\text{with } {\U}_{\tau_i} =\mat_0({X}_{\tau_i})^\top \text{ for a rank-truncated sub-TTN } X_{\tau_i}$} 		
		\Begin{
	\For{$i=1:m$ {\rm in parallel}} {
			compute the reduced SVD $\ \mat_i( \wh{C}_\tau )= \wh\bfP_{\tau_i}\mathbf{\wh\Sigma}_{\tau_i}\wh\Q_{\tau_i}^*$; \\
			set $r_{\tau_i}$ to be the smallest integer such that
			\begin{align*}
				\left( \sum\limits_{k=r_{\tau_{i}}+1}^{\wh r_{\tau_i}} \sigma_k^2 \right)^{1/2} \leq \vartheta,
			\end{align*}
			where $\sigma_k$ are the singular values in the diagonal matrix $\mathbf{\wh\Sigma}_{\tau_i}$; \\
			set $\bfP_{\tau_i}\in \C^{\wh r_{\tau_i} \times r_{\tau_i}}$ as the matrix of 
			the first $r_{\tau_i}$ columns of $\widehat{\bfP}_{\tau_i}$; \\
%
%
%
%
%
		\uIf{$\tau_i = l$ \text{\rm is a leaf}} {
			set ${\U}_{l} = \wh\U_{l} \bfP_{l}\in \C^{n_l\times r_l}$
		}
		\Else{
			set $\wt C_{\tau_{i}} = \wh C_{\tau_{i}} \times_0 {\bfP}_{\tau_i}^\top$, where $\wh C_{\tau_{i}}$ is the connection tensor of~$\wh X_{\tau_i};$
			\\ set  $\wt X_{\tau_i}$ to be the TTN in which the connection tensor in $\wh X_{\tau_i}$ 
			 is replaced with $\wt C_{\tau_{i}} $;  \\
			compute ${X}_{\tau_i} = \Theta_{\tau_i}(\wt X_{\tau_i},\vartheta)$  \quad \% {\tt recursive truncation}
			\\
			set $ {\U}_{\tau_i} = \mat_0({X}_{\tau_i})^\top$
		}
		
	}
				set
			$
                         C_\tau = \wh C_\tau  \bigtimes_{i=1}^m \bfP_{\tau_i}^* \in \C^{r_\tau\times r_{\tau_1}\times \dots\times r_{\tau_m}}
                         $
	}
\end{algorithm}
\ecl

\subsection{Recursive adaptive rank truncation} \label{subsec:truncation}
\bcl
For the truncation of a tree tensor network we perform a recursive root-to-leaves SVD-based truncation with a given tolerance $\vartheta$. For binary trees, this could be done by the rank truncation algorithms studied in \cite{G10} and \cite[Sec.\,11.4.2]{H19}. As we are not aware of a formulation and error analysis of a rank truncation algorithm for 
general (not necessarily binary) trees in the literature, we include a derivation and error analysis in the Appendix. 

In Algorithm 7 we give a formulation of the truncation algorithm as we used it in our computations.
\ecl

With the rank augmentation and truncation,  
the integrator chooses its rank adaptively in each time step. For efficiency, it is reasonable to set a maximal rank.

\subsection{Computational complexity}
As in \cite{CeLW21}, counting the required operations and the required memory yields the following result. \bcl Here
we make an assumption on the tensor-valued function $F$ (or rather on its approximation in the algorithm): For every tree tensor network $X_{\bar\tau}$, the function value $Z_{\bar\tau}=F(t,X_{\bar\tau})$ is approximated by a tree tensor network with ranks $s_\tau \le cr$ with $r=\max_\tau r_\tau$ for all subtrees $\tau\le\bar\tau$ with a moderate constant $c$ (e.g., $c=4$, as is the case in our numerical experiments in Section~\ref{sec:num}). \ecl 

\begin{lemma}\label{lem:complexity}
Let $d$ be the order of the tensor $A(t)$ (i.e., the number of leaves of the tree $\bar\tau$), $l<d$ the number of levels (i.e., the height of the tree $\bar\tau$), and let $n=\max_\ell n_\ell$ be the maximal dimension, $r=\max_\tau r_\tau$ the maximal rank and $m$ the maximal order of the connection tensors. \bcl Under the above assumption on the approximation of~$F$, \ecl  one time step of the tree tensor integrator given by Algorithms~\ref{alg:main-tau}--\ref{alg:truncation-tau} requires
\begin{itemize}
\item $O(dr(n+r^{m}))$ storage,
\item $O(ld)$ tensorizations/matricizations of matrices/tensors with $\le r^{m+1}$ entries,
\item $O(ld^2r^2(n+ r^{m}))$ arithmetical operations and
\item $O(d)$ evaluations of the function $F$,
\end{itemize}
provided the differential equations in Algorithms~\ref{alg:Phi-tau-i} and~\ref{alg:Psi-tau} are solved approximately using a fixed number of function evaluations per time step.
\end{lemma}

The bottleneck of the implementation is an efficient evaluation of the right-hand side function $F_{\bar\tau}(t,Y)\approx F(t,Y)$ of the differential equation \eqref{full-eq}.
We refer to \cite{KT14} for an efficient method of storing and applying linear operators to tree tensor networks.  The operators $F_{\tau}$ for subtrees $\tau < \bar{\tau}$ are efficiently implemented via  prolongation and restriction operators as described in \cite{CeLW21}; see also Section~\ref{subsec:subtrees} for their definition.

\subsection{Alternative TTN integrators}
\label{subsec:alt-ttn-int}
We briefly discuss relations and differences to other TTN integrators.

\smallskip\noindent
{\bf Fixed-rank TTN integrator based on the `unconventional' integrator of~\cite{CL21}.}
The rank-adaptive TTN integrator presented above extends the rank-adaptive integrator for low-rank matrices and Tucker tensors proposed and studied in \cite{CeKL22}, which in turn was conceptually based on the `unconventional' basis-update \& Galerkin low-rank integrator of \cite{CL21}. The latter integrator
 can be extended to tree tensor networks in the same way as above. This yields a fixed-rank TTN integrator which differs from the presented algorithm only in that the matrices of doubled dimension $\bigl( Y_{l}(t_1)^\top,Y_{l}^{0,\top} \bigr)$ and $\bigl( \mat_0(\wh C_{\tau_i}^1)^\top , \mat_0(\wh {C}_{\tau_i}^0)^\top \bigr) $ in Algorithm~\ref{alg:Phi-tau-i}, which contain new and old values, are replaced by taking instead only the new values $Y_{l}(t_1)^\top$ and $\mat_0(\wh C_{\tau_i}^1)^\top$. The rank truncation $\Theta_\tau$ is then not needed. We found, however, that the
integrator presented above in Algorithms~\ref{alg:main-tau}--\ref{alg:truncation-tau}, even when taken as a fixed-rank integrator in which the augmented ranks are always truncated to the original ranks, is more accurate and has better conservation properties. Using in the Galerkin basis both the old and new values instead of only the new values appears to be distinctly favourable.

\smallskip\noindent
{\bf Comparison with the TTN integrator of \cite{CeLW21} based on the projector-splitting integrator.}
A conceptually different fixed-rank TTN integrator was proposed and studied in \cite{CeLW21}. That integrator is based on the projector-splitting integrator for low-rank matrices \cite{LuO14}, which was previously extended to Tucker tensors \cite{Lu15,LuVW18} and tensor trains / matrix product states \cite{LuOV15,HaLOVV16}. The TTN integrator of   \cite{CeLW21} has the same robustness to small singular values as the algorithm presented here. It differs in that some substeps use evolutions backward in time, which are problematic for strongly dissipative problems. It conserves, however, norm and total energy of Schr\"odinger equations exactly up to errors in the integration of the matrix and connection tensor differential equations in the substeps. The TTN integrator of   \cite{CeLW21}
has a more serial structure and thus has less parallelism, and --- of principal interest here --- it does not so easily generalize to a rank-adaptive integrator. We mention, however, that elaborate suggestions for rank-adaptive extensions of the projector-splitting integrator were made in \cite{DeRV21,DuC21,YaW20} in the case of tensor trains / matrix product states.

\section{Exactness property and robust error bound}	  \label{sec:robust}

The rank-adaptive TTN integrator shares fundamental robustness properties with the TTN integrator of~\cite{CeLW21}. 

\subsection{Exactness property} 
The exactness property of Theorem~5.1 in~\cite{CeLW21} states that the TTN integrator exactly reproduces time-dependent tree tensor networks that are of the tree rank used by the integrator. This holds true also for the rank-adaptive TTN integrator without truncation (and also with truncation provided the truncation tolerance $\vartheta$ is sufficiently small). This is shown by using the exactness property of the rank-adaptive integrator for matrices and Tucker tensors as given in~\cite{CeKL22} in a  proof by induction over the height of the tree in the same way as in the proof of Theorem~5.1 in~\cite{CeLW21}.

\subsection{Robust error bound} 
\label{subsec:error-bound}
\bcl
The error bound of Theorem~6.1 in \cite{CeLW21}, which is independent of small singular values of matricizations of the connection tensors, extends to the rank-adaptive TTN integrator with an extra term proportional to the truncation tolerance divided by the stepsize,
$\vartheta/h$, in the $O(h+\eps)$ error bound as in Theorem~2 of~\cite{CeKL22}. 
This is shown by induction over the height of the tree in the same way as in \cite{CeLW21}, using the error bound of the rank-adaptive integrator for matrices and Tucker tensors given in~\cite{CeKL22}, which in turn is based on \cite{KiLW16} and relies on the exactness property.

\bcl
We do not include a detailed proof but give a precise statement of the error bound and its assumptions. The assumptions are those of 
Theorem~6.1 in \cite{CeLW21}, but with the added complication that due to the changing ranks, the TTN manifold is different in every time step.
For a tree $\tau=(\tau_1,\ldots,\tau_m)$ we use the notation $\mathcal{V}_\tau=\C^{r_\tau\times n_{\tau_1}\times \dots \times n_{\tau_m}}$ for the corresponding tensor space 
and $\M_\tau^k=\M_\tau((n_\tau)_{\tau\le\bar\tau},(r_\tau^k)_{\tau\le\bar\tau})$ for the TTN manifold in the $k$th time step; cf.~Section~\ref{subsec:subtrees}. We set $\mathcal{V}=\mathcal{V}_{\bar\tau}$ and
$\mathcal{M}^k=\mathcal{M}_{\bar\tau}^k$ for the full tree $\bar\tau$. We make the following assumptions; cf.~\cite{CeKL22,CeLW21}:

1. We assume that $F:[0,t^*]\times \mathcal{V} \to \mathcal{V}$ is Lipschitz continuous and bounded,
\begin{align}
\label{L-bound}
 & \|F(t,Y) - F(t,\widetilde{Y})\| \le L\|Y-\widetilde{Y}\| & \text{ for all}\ Y,\widetilde{Y} \in \mathcal{V},  
 \\[1mm]
 \label{B-bound}
 & \| F(t,Y) \| \le B & \text{ for all}\ Y \in \mathcal{V}. 
\end{align}\label{eq:F-bound}
Here and in the following, the chosen norm $\|\cdot\|$ is the tensor Euclidean norm. As usual in the numerical analysis of ordinary differential equations, this could be weakened to a local Lipschitz condition and local bound in a neighborhood of 
the exact solution $A(t)$ of the tensor differential equation \eqref{full-eq} to the initial data $A(t_0)=A^0\in\mathcal{V}$. 

2. For $t$ near $t_k=kh$ and $Y$ near the exact solution $A(t)$, we assume that the function value $F(t,Y)$ is in the tangent space $\mathcal{T}_Y\M^k$ up to a small remainder: with $P^k_Y$ denoting the orthogonal
projection onto $\mathcal{T}_Y\M^k$, we assume that for some $\eps>0$,
\begin{equation}\label{eps}
\| F(t,Y) - P^k_Y F(t,Y) \| \le \eps 
\end{equation}
for all $(t,Y)\in [t_k,t_{k+1}]\times \M^k$ \bcl in some ball $\| Y \| \le \rho$, where it is assumed that the exact solution $A(t)$, $0\le t \le t^*$, has a bound that is strictly smaller than $\rho$.

3. The initial value $A^0\in \mathcal{V}$ and the starting value $Y^0\in\M^0$ of the numerical method are assumed to be $\delta$-close:
\begin{equation} \label{init-err}
  \| Y^0-A^0 \| \le \delta.
\end{equation}

\begin{theorem}[Error bound]
  \label{thm:error}
  Under assumptions $1.$--$\,3.$, the error of the numerical approximation $Y^k\in \M^k$ at $t_k=kh$, obtained with $k$ time steps of the rank-adaptive tree tensor network integrator with step size $h>0$ and rank-truncation tolerance $\vartheta$, is bounded by
  \[
    \|Y^k - A(t_k)\| \le c_0\delta + c_1 \eps + c_2 h + c_3 \vartheta/h \qquad\text{for }\ t_k \le t^*,
  \]
where $c_i$ depend only on $L$, $B$, $t^*$, and the tree $\bar\tau$. \bcl This holds true provided that $\delta,\eps,h$ and $\vartheta/h$ are so small that the above error bound guarantees that  $ \|Y^k\| \le \rho$.\ecl
\end{theorem}

The important fact is that the constants $c_i$ are independent of small singular values of matricizations of the connection tensors. This would not be possible by applying a standard integrator to the system of differential equations for the TTN basis matrices and connection tensors that is equivalent to \eqref{proj-eq}, as derived in \cite{WaT03}, since this system becomes near-singular in the case of small singular values.

\ecl

\section{Structure-preserving properties}  \label{sec:cons}
We show that the augmented TTN integrator given by Algorithms~\ref{alg:main-tau}--\ref{alg:Psi-tau}
has remarkable conservation properties: It preserves norm and energy for Schr\"odinger equations, and it diminishes the energy for gradient systems.
 The rank-adaptive TTN integrator, which truncates the result of the rank-augmenting integrator, then has corresponding near-conservation properties
up to a multiple of the truncation tolerance $\vartheta$ (and up to errors in the integration of the low-dimensional differential equations in the substeps, which we disregard here). 
The proof of such properties relies on Lemma~\ref{lem:Y-init} below and Theorem~A.1 in the Appendix. The following results extend conservation results 
in~\cite{CeKL22} from the dynamical low-rank approximation of matrix differential equations to the more intricate general TTN case.

In this section we write $\wh Y_\tau^1$  for the rank-augmented result $\wh X_\tau$ after a time step with step size $h$, and
$Y_\tau^1$ for the rank-truncated result $X_\tau$, which is used as the starting value for the next time step. (We mention that for $\tau\ne \bar\tau$ this notation is not consistent with the notation $\wh Y_\tau^1$ used in Algorithms~\ref{alg:main-tau} and~\ref{alg:Phi-tau-i}, but here it serves us well in comparing quantities at times $t_0$ and $t_1$.)

\subsection{Starting tensors}

We first show that for each subtree $\tau=(\tau_1,\dots,\tau_m)$ of $\bar\tau$, the given starting tensor
$$
Y_\tau^0 = C_\tau^0 \times_0 \I_\tau \bigtimes_{i=1}^m \U_{\tau_i}^0
$$
coincides with the corresponding tensor that has the augmented connection tensor $\wh C_\tau^0$ and basis matrices $\wh \U_{\tau_i}$ constructed in Algorithm~\ref{alg:Phi-tau-i},
$$
\wh Y_\tau^0 = \wh C_\tau^0 \times_0 \I_\tau \bigtimes_{i=1}^m \wh \U_{\tau_i},
$$
which is actually used in Algorithm~\ref{alg:Psi-tau}. The following is a key lemma for the conservation properties proved later in this section.

\begin{lemma} \label{lem:Y-init}
$\wh Y_\tau^0 = Y_\tau^0$.
\end{lemma}

\begin{proof} 
Since $ \wh C_\tau^0 =   C_\tau^0 \bigtimes_{i=1}^m \wh \bfM_{\tau_i}$ in
Algorithm~\ref{alg:Psi-tau} with $\wh\bfM_{\tau_i}=\wh \U_{\tau_i}^*\U_{\tau_i}^0$,
we have
\begin{align*}
\wh Y_\tau^0 &= \bigl( C_\tau^0 \bigtimes_{i=1}^m \wh \U_{\tau_i}^*\U_{\tau_i}^0 \bigr)
\times_0 \I_{\tau} \bigtimes_{i=1}^m \wh \U_{\tau_i}
\\
&= C_\tau^0 \times_0 \I_{\tau} \bigtimes_{i=1}^m \wh\U_{\tau_i}\wh \U_{\tau_i}^* \U_{\tau_i}^0.
\end{align*}
We will show by induction over the height of the tree that
\begin{equation}\label{ranges-U}
\text{Range}(\U_{\tau_i}^0) \subseteq \text{Range}(\wh\U_{\tau_i}).
\end{equation}
Since $\wh\U_{\tau_i}\wh\U_{\tau_i}^*$ is the orthogonal projection onto $\text{Range}(\wh\U_{\tau_i})$, this implies
$$
\wh\U_{\tau_i}\wh\U_{\tau_i}^* \U_{\tau_i}^0 = \U_{\tau_i}^0, \qquad i=1,\dots,m,
$$
and hence we obtain 
$$
\wh Y_\tau^0 = C_\tau^0 \times_0 \I_{\tau} \bigtimes_{i=1}^m \U_{\tau_i}^0 = Y_\tau^0.
$$
It remains to prove \eqref{ranges-U}. If $\tau_i=l$ is a leaf, then this is obvious from the definition of $\wh\U_l$ in Algorithm~\ref{alg:Phi-tau-i}.
Else we write the tree $\tau_i$ as $\sigma=(\sigma_1,\dots,\sigma_k)$ with direct subtrees $\sigma_j$ and use the induction hypothesis that 
$\text{Range}(\U_{\sigma_j}^0) \subseteq \text{Range}(\wh\U_{\sigma_j})$ for all $j=1,\dots,k$. This implies that
$$
\U_{\sigma_j}^0 = \wh\U_{\sigma_j} \wh\U_{\sigma_j}^* \U_{\sigma_j}^0 = \wh\U_{\sigma_j} \wh\bfM_{\sigma_j}.
$$
By the construction of the tree tensor network, we have
$$
\U_\sigma^0 = \mat_0(Y_\sigma^0)^\top \quad\text{ with }\quad
Y_\sigma^0 = C_\sigma^0 \times \I_\sigma \bigtimes_{j=1}^k \U_{\sigma_j}^0.
$$
So we obtain from the matricization formula \eqref{eq:mat_tucker} that
\begin{align*}
\U_\sigma^0 &= \bigotimes_{j=1}^k \U_{\sigma_j}^0 \mat_0(C_\sigma^0)^\top
= \bigotimes_{j=1}^k \wh\U_{\sigma_j} \wh\bfM_{\sigma_j} \mat_0(C_\sigma^0)^\top
\\
&= \bigotimes_{j=1}^k \wh\U_{\sigma_j} \mat_0\bigl(C_\sigma^0\bigtimes_{j=1}^k  \wh\bfM_{\sigma_j} \bigr)^\top.
\end{align*}
Since we have $ \wh C_\sigma^0 =   C_\sigma^0 \bigtimes_{j=1}^k \wh \bfM_{\sigma_j}$ in
Algorithm~\ref{alg:Psi-tau}, we finally obtain
$$
\U_\sigma^0 = \bigotimes_{j=1}^k \wh\U_{\sigma_j}  \mat_0(\wh C_\sigma^0)^\top.
$$
On the other hand,
in Algorithm~\ref{alg:Phi-tau-i} we construct
$$
\wh \U_\sigma =  \bigotimes_{j=1}^k \wh\U_{\sigma_j}  \mat_0(\wh C_\sigma)^\top,
$$
where the columns of $\mat_0(\wh C_\sigma)^\top$ form an orthogonal basis of the range of the augmented matrix
$(\mat_0(\wh C_\sigma^1)^\top,\mat_0(\wh C_\sigma^0)^\top)$. We therefore have 
$$
\text{Range}(\mat_0(\wh C_\sigma^0)^\top) \subseteq \text{Range}(\mat_0(\wh C_\sigma)^\top),
$$
and in view of the above formulas for $\U_\sigma^0$ and $\wh \U_\sigma$, this implies the relation \eqref{ranges-U} for $\sigma=\tau_i$.
\end{proof}

\subsection{Norm conservation}
If the function $F$ on the right-hand side of the tensor differential equation \eqref{full-eq} satisfies,
with $\langle\cdot,\cdot\rangle$ denoting the Euclidean inner product of vectorizations,
\begin{equation} \label{norm-cons}
	\Re\!\langle Y, F(t,Y) \rangle = 0 \ \ \ \text{for all}  \ Y \in \C^{n_1\times\dots\times n_d} \ \text{and for all} \ t, 
\end{equation}
then the Euclidean norm of every solution $A(t)$ of  \eqref{full-eq} is conserved: $\| A(t) \| = \| A(t_0) \|$ for all $t$. Norm conservation also holds true for the rank-augmented TTN integrator, as is shown by the following result.

\begin{theorem} If $F$ satisfies \eqref{norm-cons}, then a step of the rank-augmented TTN integrator preserves the norm: for every stepsize $h$ and for every subtree $\tau$ of~$\bar\tau$,
$$
\| \wh Y_\tau^1 \| = \| Y_\tau^0 \|.
$$
\end{theorem}

By Theorem~A.1, this implies near-conservation of the norm up to a multiple of the truncation tolerance $\vartheta$ for the rank-adaptive TTN integrator:
$$
\| Y_\tau^0 \| - c_\tau \vartheta \le \| Y_\tau^1 \| \le \| Y_\tau^0 \| ,
$$
with $c_\tau =\| C_{\tau} \| (d_{\tau}-1) +1$, as in Theorem~A.1.

\begin{proof} For every subtree $\tau=(\tau_1,\dots,\tau_m) \le \bar\tau$, the reduced nonlinear operator $F_\tau$ is defined in Section~\ref{subsec:subtrees} by the recursion \eqref{F-tau-i} from the root to the leaves, starting from $F_{\bar\tau}=F$. We first observe that 
\begin{equation} \label{norm-cons-tau}
	\Re\langle Y_\tau, F_\tau(t,Y_\tau) \rangle = 0 \ \ \ \text{for all}  \ Y_\tau \in \mathcal{V}_\tau \ \text{and for all} \ t, 
\end{equation}
which follows by induction from the root to the leaves, noting
$$
\langle Y_{\tau_i}, F_{\tau_i}(t,Y_{\tau_i}) \rangle = \langle Y_{\tau_i}, \pi_{\tau,i}^\dagger F_\tau(t,\pi_{\tau,i}Y_{\tau_i}) \rangle =
\langle \pi_{\tau,i}Y_{\tau_i},  F_\tau(t,\pi_{\tau,i}Y_{\tau_i}) \rangle 
$$
and $\pi_{\tau,i}Y_{\tau_i}\in \mathcal{V}_\tau$.

We now turn to the differential equation in Algorithm~\ref{alg:Psi-tau}. We have, for $k=0$ and $k=1$, 
$$
\widehat{Y}_\tau^k = \widehat{C}_\tau^k \bigtimes_{i=1}^m \widehat{\U}_{\tau_i} \quad\text{ and hence}\quad
 \|\widehat{Y}_\tau^k \| = \| \widehat{C}_\tau^k \|,
$$
since $\widehat{\U}_{\tau_i}$ 
have orthonormal columns for $i=1,\dots,m$.
We recall that $\widehat{C}_\tau^1$ is the solution at $t_1$ of the differential equation
\begin{align*}
	\dot{\widehat{C}}_\tau (t) = F_\tau \left( t, \widehat{C}_\tau(t)\bigtimes_{i=1}^m \widehat{\U}_{\tau_i} \right) \bigtimes_{i=1}^m \widehat{\U}_{\tau_i}^{*}
\end{align*}
with initial value $\wh C_\tau(t_0)= \wh C_\tau^0$.
We then get
\begin{align*}
	\frac{1}{2} \frac{d}{dt} \| \widehat{C}_\tau(t) \|^2 &= \Re \langle \widehat{C}_\tau(t), \dot{\widehat{C}}_\tau(t) \rangle 
	= \Re\langle \widehat{C}_\tau(t), F_\tau ( t, \widehat{C}_\tau(t) \bigtimes_{i=1}^m \widehat{\U}_{\tau_i} ) \bigtimes_{i=1}^m \widehat{\U}_{\tau_i}^{*} \rangle \\
	&= \Re\langle \widehat{C}_\tau(t) \bigtimes_{i=1}^m \widehat{\U}_{\tau_i}, F_\tau ( t, \widehat{C}_\tau(t) \bigtimes_{i=1}^m \widehat{\U}_{\tau_i})  \rangle = 0 \quad\text{ by \eqref{norm-cons-tau}.}
\end{align*}
Therefore we have 
\begin{align*}
	\| \widehat{Y}_\tau^1 \| = \| \widehat{C}_\tau^1 \| = \| \widehat{C}_\tau (t_1) \| = \| \widehat{C}_\tau (t_0) \| = \| \wh C_\tau^0\| = \| \widehat{Y}_\tau^0 \| = \| Y_\tau^0 \|,
\end{align*}
where the last equality comes from Lemma \ref{lem:Y-init}. 
\end{proof}

\subsection{Energy conservation for Schr\"odinger equations}
Consider the tensor Schr\"odinger equation 
\begin{align}
	\mathrm{i}\dot{A}(t) = H [ A(t) ], \label{Schr_Eq}
\end{align}
with a Hamiltonian $H: \C^{n_1 \times \cdots \times n_d} \rightarrow \C^{n_1 \times \cdots \times n_d}$ that is linear and self-adjoint, i.e.,
	$\langle H[Y],Z \rangle = \langle Y, H[Z] \rangle$ for all
	 $Y,Z \in \C^{n_1 \times \cdots \times n_d}$.
The energy
\begin{align*}
	\mathrm{E}(Y) = \langle Y, H[Y] \rangle
\end{align*}
is preserved along solutions of \eqref{Schr_Eq}: $\mathrm{E}(A(t))=\mathrm{E}(A(t_0))$ for all $t$.

\begin{theorem} The rank-augmented TTN integrator preserves the energy: for every stepsize $h$,
$$
\mathrm{E}(\wh Y_{\bar\tau}^1)  = \mathrm{E}(Y_{\bar\tau}^0) .
$$
\end{theorem}
By  the Cauchy--Schwarz inequality and Theorem~A.1,  this implies for the rank-adaptive TTN integrator that
\begin{align*}
	\vert \mathrm{E}(Y_{\bar\tau}^1 ) - \mathrm{E}({Y}_{\bar\tau}^0 ) \vert &= \vert \mathrm{E}(Y_{\bar\tau}^1 ) - \mathrm{E}(\wh{Y}_{\bar\tau}^1 ) \vert 
	 =\vert \text{Re} \langle Y_{\bar\tau}^1 -\wh{Y}_{\bar\tau}^1, H[Y_{\bar\tau}^1 +\wh{Y}_{\bar\tau}^1] \rangle \vert 
	\\
	&\leq c_{\bar\tau} \vartheta \, \| H[Y_{\bar\tau}^1 +\wh{Y}_{\bar\tau}^1] \|.
\end{align*}

\begin{proof} In Algorithm~\ref{alg:Psi-tau} we have for
	 $\bar\tau=(\tau_1,\dots,\tau_m)$ and $\wh{Y}_{\bar\tau}(t) = \wh{C}_{\bar\tau} (t) \bigtimes_{j=1}^m \wh{\U}_{\tau_j}$
\begin{align*}
	\frac{d}{dt} \mathrm{E}(\wh{Y}_{\bar\tau}(t) ) &= 2\,\text{Re}\langle H[\wh{Y}_{\bar\tau}(t)], \dot{\wh{C}}_{\bar\tau} (t) \bigtimes_{j=1}^m \wh{\U}_{\tau_j} \rangle \\
	&= 2\,\text{Re}\langle H[\wh{Y}_\tau(t)]  \bigtimes_{j=1}^m \wh{\U}_{\tau_j}^* , \dot{\wh{C}}_{\bar\tau} (t) \rangle \\
	&= 2\,\text{Re}\langle H[\wh{Y}_\tau(t)]  \bigtimes_{j=1}^m \wh{\U}_{\tau_j}^* , -\mathrm{i}H[\wh{Y}_{\bar\tau}(t)] \bigtimes_{j=1}^m \wh{\U}_{\tau_j}^* \rangle \\
	&= -2\,\text{Re} \ \mathrm{i}\, \| H[\wh{Y}_\tau(t)]  \bigtimes_{i=j}^m \wh{\U}_{\tau_j}^* \|^2 = 0.
\end{align*}
Using $\wh{Y}_\tau^1 = \wh{Y}_\tau(t_1)$ and Lemma \ref{lem:Y-init}, we obtain
\begin{align*}
	\mathrm{E}(\wh{Y}_\tau^1) = \mathrm{E}(\wh{Y}_\tau(t_1)) = \mathrm{E}(\wh{Y}_\tau(t_0) ) = \mathrm{E}(\wh Y_\tau^0)=\mathrm{E}(Y_\tau^0),
\end{align*}
which is the stated result.
\end{proof}

We note further that for each subtree $\tau\le \bar\tau$, the reduced operator $F_\tau$ of Section~\ref{subsec:subtrees} that corresponds to $F_{\bar\tau}(Y)=-\mathrm{i} H[Y]$, is again a Schr\"odinger operator $F_\tau(Y_\tau)=-\mathrm{i} H_\tau[Y_\tau]$ with a self-adjoint linear operator $H_\tau:\mathcal{V}_\tau \to \mathcal{V}_\tau$, since  we have recursively $H_{\tau_i}= \pi_{\tau,i}^\dagger H_\tau \pi_{\tau,i}$ for the $i$th subtree $\tau_i$ of $\tau$. Hence, the integrator preserves the energy on the level of each subtree.

\subsection{Energy decay for gradient systems}
We now consider the case where the tensor differential equation \eqref{full-eq} is a gradient system 
$$
\dot A(t) = -\nabla \mathrm{E}(A(t)) \quad\text{ with a given function}\ \mathrm{E}:\R^{n_1 \times \cdots \times n_d} \rightarrow \R.
$$
Along every solution, we then have the energy decay
\begin{align*}
	\frac{d}{dt} \mathrm{E}(A(t)) = \langle \nabla \mathrm{E}(A(t)),\dot{A}(t) \rangle = - \| \nabla \mathrm{E}(A(t)) \|^2.
\end{align*}

\begin{theorem} The rank-augmented TTN integrator diminishes the energy: for every stepsize $h>0$,
$$
\mathrm{E}(\wh Y_{\bar\tau}^1)  \le  \mathrm{E}(Y_{\bar\tau}^0)  - \alpha^2 h,
$$
where $\alpha =  \min_{0 \leq \mu \leq 1} \| \nabla \mathrm{E}(\wh{Y}_{\bar\tau}(t_0 + \mu h)) \bigtimes_{i=1}^{m} \wh{\U}_{\tau_i}^* \| = \| \nabla \mathrm{E}({Y}_{\bar\tau}^0) \bigtimes_{i=1}^{m} \wh{\U}_{\tau_i}^* \| + O(h)$.
\end{theorem}

By the mean value theorem and Theorem \ref{thm:err-trunc},  this implies for the rank-adaptive TTN integrator that
\begin{align*}
	\mathrm{E}(Y_{\bar\tau}^1 ) \leq \mathrm{E}(\wh{Y}_{\bar\tau}^1 ) + \beta c_{\bar\tau} \vartheta \le \mathrm{E}(Y_{\bar\tau}^0)  - \alpha^2 h + \beta c_{\bar\tau} \vartheta
\end{align*}
with $\beta = \text{max}_{0\leq \mu \leq 1} \| \nabla \mathrm{E}( \mu Y_{\bar\tau}^1  + (1-\mu)\wh{Y}_{\bar\tau}^1 ) \|=\| \nabla \mathrm{E}(Y_{\bar\tau}^1) \|+O(\vartheta)$.

\begin{proof} The proof has the same structure as the previous proof.
	In Algorithm~\ref{alg:Psi-tau} we have for $\bar\tau=(\tau_1,\dots,\tau_m)$ and $\wh{Y}_{\bar\tau} (t) = \wh{C}_{\bar\tau}(t) \bigtimes_{i=1}^m \wh{\U}_{\tau_i}$
	\begin{align*}
		\frac{d}{dt} \mathrm{E}(\wh{Y}_{\bar\tau}(t)) &= \langle \nabla \mathrm{E}(\wh{Y}_{\bar\tau}(t)), \dot{\wh{C}}_{\bar\tau}(t) \bigtimes_{i=1}^{m} \wh{\U}_{\tau_i} \rangle \\
		&= \langle \nabla \mathrm{E}(\wh{Y}_{\bar\tau}(t)) \bigtimes_{i=1}^{m} \wh{\U}_{\tau_i}^*, \dot{\wh{C}}_{\bar\tau}(t) \rangle \\
		&= \langle \nabla \mathrm{E}(\wh{Y}_{\bar\tau}(t)) \bigtimes_{i=1}^{m} \wh{\U}_{\tau_i}^*, - \nabla \mathrm{E}(\wh{Y}_{\bar\tau}(t)) \bigtimes_{i=1}^{m} \wh{\U}_{\tau_i}^* \rangle \\
		&= - \| \nabla \mathrm{E}(\wh{Y}_{\bar\tau}(t)) \bigtimes_{i=1}^{m} \wh{\U}_{\tau_i}^* \|^2 \leq -\alpha^2,
	\end{align*}
with $\alpha = \text{min}_{0 \leq \mu \leq 1} \| \nabla \mathrm{E}(\wh{Y}_{\bar\tau}(t_0 + \mu h)) \bigtimes_{i=1}^{m} \wh{\U}_{\tau_i}^* \|$. By Lemma~\ref{lem:Y-init} we have $\wh{Y}_{\bar\tau}^0 = Y_{\bar\tau}^0$ and with $\wh{Y}_{\bar\tau}^1 = \wh{Y}_{\bar\tau}(t_1)$, we obtain
\begin{align*}
	\mathrm{E}(\wh{Y}_{\bar\tau}^1) = \mathrm{E}(\wh{Y}_{\bar\tau}(t_1)) \leq \mathrm{E}(\wh{Y}_{\bar\tau}(t_0)) - \alpha^2 h = \mathrm{E}(Y_{\bar\tau}^0) - \alpha^2 h,
\end{align*}
which is the stated result.
\end{proof}

We note further that for each subtree $\tau\le \bar\tau$, the reduced operator $F_\tau$ of Section~\ref{subsec:subtrees} that corresponds to $F_{\bar\tau}(Y)=-\nabla \mathrm{E}[Y]$, is again a gradient $F_\tau(Y_\tau)=-\nabla \mathrm{E}_\tau[Y_\tau]$ with an energy function $\mathrm{E}_\tau:\mathcal{V}_\tau \to \R$, defined recursively as $\mathrm{E}_{\tau_i}= \mathrm{E}_\tau \circ \pi_{\tau,i}$ for the $i$th subtree $\tau_i$ of $\tau$. Hence, the integrator dissipates the energy on the level of each subtree.

\section{Numerical experiments}   \label{sec:num}
We illustrate the use of the rank-adaptive TTN integrator for a problem in quantum physics.
 \bcl We show the  behaviour of the numerical error of a physical observable of interest (the magnetization in a quantum spin system) and the numerical errors in the conservation of norm and energy,
 and we compare the ranks and the total numbers of evolving parameters in the system as selected by the rank-adaptive algorithm for two different types of trees: binary trees of minimal and maximal height (the latter correspond to tensor trains/matrix product states).\ecl
 
We consider a basic quantum spin system, the Ising model in a transverse field with next neighbor interaction; 
see e.g.~\cite{S73}. 
Consider the discrete Schr\"odin\-ger equation
\begin{align} \label{Ising}
	\mathrm{i}\,\partial_t  \psi  &= H  \psi  \quad\text{with }\ 
%
	H = -\Omega\sum_{k=1}^d \sigma_x^{(k)} - \sum_{k=1}^{d-1} \sigma_z^{(k)}\sigma_z^{(k+1)},
\end{align}
where $\Omega \geq 0$, $\sigma_x={\small\begin{pmatrix} 0 & 1 \\ 1 & 0 \end{pmatrix}}$ and $\sigma_z={\small\begin{pmatrix} 1 & 0 \\ 0 &  -1\end{pmatrix}}$ are the first and third Pauli matrices, respectively, and $\sigma^{(j)} = (\I\otimes \dots  \otimes \I \otimes \sigma \otimes \I \otimes\dots  \otimes \I)$, with $\sigma$ acting on the $j$th particle. 
We take the initial value
$\psi^0  = \bigotimes_{k=1}^{d} (1,0)^\top\in (\C^2)^{d}$.

\subsection{Numerical error behavior for an observable} An observable of interest is the magnetization in $z$-direction, defined by
\begin{align*}
	\langle \psi \vert M \vert \psi \rangle = \frac{1}{d} \sum_{k=1}^d \langle \psi \vert \sigma_z^{(k)} \vert \psi \rangle.
\end{align*}
We solve (\ref{Ising}) and compute the magnetization over time, i.e., $\langle M \rangle (t)= \langle \psi(t) \vert M \vert \psi(t) \rangle$, for $d=10$, $\Omega = 1$, using the rank-adaptive integrator on a binary tree of height 4 with step-size $h=0.01$ and different tolerance parameters $\vartheta$; see Figure~\ref{fig:mag_error}.
We used the classical fourth-order Runge--Kutta method for solving the  low-dimensional differential equations appearing within the subflows $\Phi_\tau^{(i)}$ and $\Psi_\tau$. This gave far better results than using just the explicit Euler method and also somewhat better results than using the second-order Heun method.
\begin{figure}[H]
	\centering
	\includegraphics[trim ={40mm 0mm 0 0mm},clip,scale=0.3]{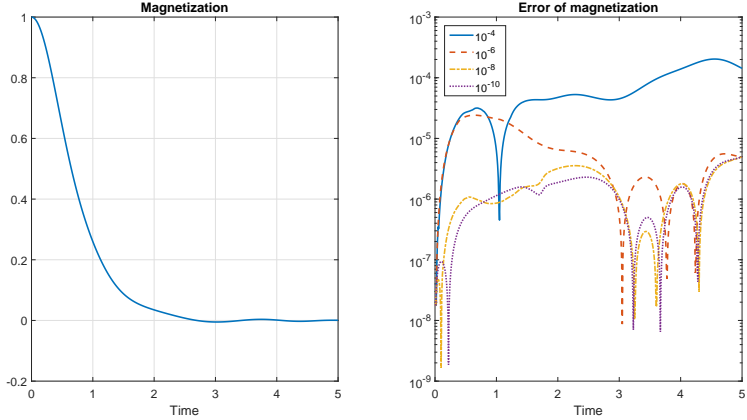}
    \caption{Magnetization $\langle M\rangle (t)$ (left) and its numerical error (right) as functions of time for different tolerance parameters $\vartheta$. }
    \label{fig:mag_error}
\end{figure}
\bcl On the right we see the numerical error of magnetization in a logarithmic scale. The reference solution is obtained by computing $\exp(-ih H)$ via diagonalization of $H$ (which is still feasible for a system of this size). The resulting matrix gives us the numerically exact time propagation via $X_{k+1} =  \exp(-ih H) X_k$. \ecl Note that decreasing $\vartheta$ further does not lead to better results, as we still have a time discretization error, which is of the order $\mathcal{O}(h)$ by the bound of Section~\ref{subsec:error-bound}. The pure time discretization error was actually observed to behave as $\mathcal{O}(h^2)$ in numerical tests but this is beyond our current theoretical understanding. 

\subsection{Near-conservation of norm and energy}
Further we illustrate the conservation properties of the rank-adaptive integrator. 
We again consider the differential equation (\ref{Ising}) with $d=10$, $\Omega = 1$, using the time step-size $h=0.01$ and the tolerance parameter $\vartheta =10^{-8}$.
\begin{figure}[H]
	\centering
	\includegraphics[trim ={40mm 0mm 0 0mm},clip,scale=0.35]{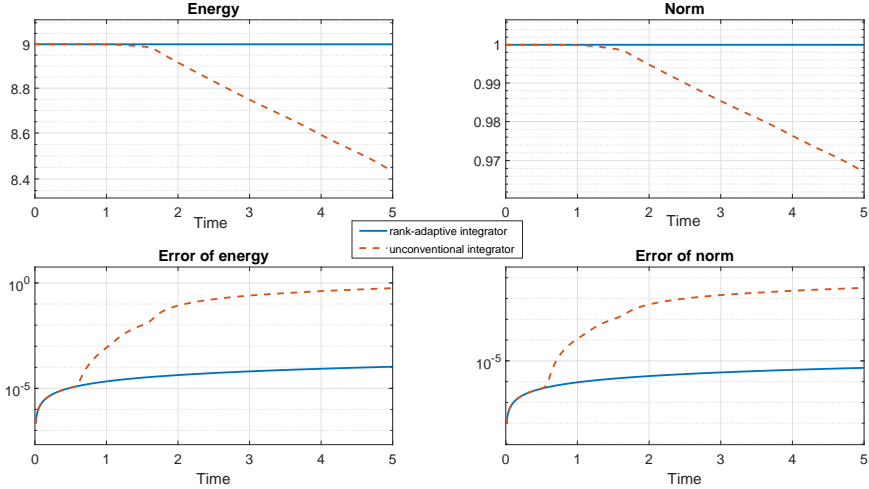}
    \caption{Conservation of energy and norm up to final time $T=5$ for the low-rank approximations generated by the rank-adaptive TTN integrator (solid line) and the fixed-rank TTN integrator (dashed line) from Section \ref{subsec:alt-ttn-int}.}
    \label{fig:conservation}
\end{figure}
In Figure \ref{fig:conservation} it is seen that the proposed rank-adaptive integrator conserves the energy and the norm very well (up to the order of the tolerance parameter $\vartheta$ according to our theory). Similar results are obtained also when choosing larger~$\vartheta$. In contrast, the TTN integrator based on the `unconventional' integrator of ~\cite{CL21} (see Section~\ref{subsec:alt-ttn-int})
has a significant drift in norm and energy.  

\subsection{Comparison of different tree structures}
Each tensor can be approximated by TTNs on different trees. We compare the behavior of the rank-adaptive integrator applied to TTNs on binary trees of minimal and maximal height, see Figure \ref{fig:binary_tree}. TTNs on binary trees of maximal height are matrix product states/tensor trains, which have become a standard computational tool in quantum physics; see e.g.\;\cite{Vi03,Scho11,Pa-etal19,CiPSV21} and references therein. \bcl For the same number $d$ of leaves, both trees have $2d-1$ vertices.\ecl
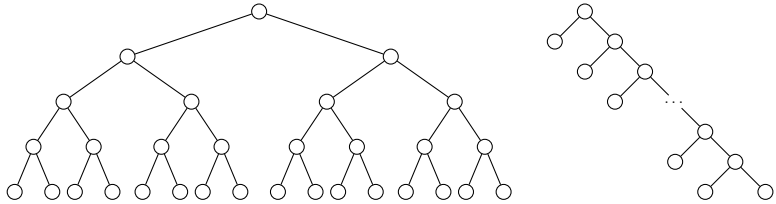
\begin{figure}[H]
	\begin{center}
	\begin{tikzpicture}[every node/.style={sibling distance=10mm},level 1/.style={sibling distance=35mm},level 2/.style={sibling distance=17mm},level 3/.style={sibling distance=8mm},level 4/.style={sibling distance=5mm},level distance = 6mm]
		\node[circle,draw,scale=0.6]  { }
		child { node[circle,draw,scale=0.6] {} 
			child{ node[circle,draw,scale=0.6] {} 
				child{ node[circle,draw,scale=0.6] {}
					child{ node[circle,draw,scale=0.6]{} }
					child{ node[circle,draw,scale=0.6]{} }}
				child{ node[circle,draw,scale=0.6] {}
					child{ node[circle,draw,scale=0.6]{} }
					child{ node[circle,draw,scale=0.6]{} }}	
					}
			child{ node[circle,draw,scale=0.6] {} 
				child{ node[circle,draw,scale=0.6] {}
					child{ node[circle,draw,scale=0.6]{} }
					child{ node[circle,draw,scale=0.6]{} }}
				child{ node[circle,draw,scale=0.6] {}
					child{ node[circle,draw,scale=0.6]{} }
					child{ node[circle,draw,scale=0.6]{} }}} }
		child { node[circle,draw,scale=0.6] {} 
			child{ node[circle,draw,scale=0.6] {} 
				child{ node[circle,draw,scale=0.6] {}
					child{ node[circle,draw,scale=0.6]{} }
					child{ node[circle,draw,scale=0.6]{} }}
				child{ node[circle,draw,scale=0.6] {}
					child{ node[circle,draw,scale=0.6]{} }
					child{ node[circle,draw,scale=0.6]{} }}	
					}
			child{ node[circle,draw,scale=0.6] {} 
				child{ node[circle,draw,scale=0.6] {}
					child{ node[circle,draw,scale=0.6]{} }
					child{ node[circle,draw,scale=0.6]{} }}
				child{ node[circle,draw,scale=0.6] {}
					child{ node[circle,draw,scale=0.6]{} }
					child{ node[circle,draw,scale=0.6]{} }}}};
	\end{tikzpicture}
	\quad
	\begin{tikzpicture}[every node/.style={sibling distance=2mm},,level 1/.style={sibling distance=8mm},level 2/.style={sibling distance=8mm},level 3/.style={sibling distance=8mm},level 4/.style={sibling distance=8mm},level 5/.style={sibling distance=8mm},level 6/.style={sibling distance=8mm},level distance = 4mm]
		\node[circle,draw,scale=0.6]  { }
			child{ node[circle,draw,scale=0.6]{} }
			child{ node[circle,draw,scale=0.6]{}{
				child{ node[circle,draw,scale=0.6]{} }
				child{ node[circle,draw,scale=0.6]{} {
					child{ node[circle,draw,scale=0.6]{} }
					child{ node[,scale=0.6]{$\dots$} {
						child[white,scale=0.6]{}
						child{ node[circle,draw,scale=0.6]{} {
							child{ node[circle,draw,scale=0.6]{} }
							child{ node[circle,draw,scale=0.6]{}{
								child{ node[circle,draw,scale=0.6]{} }
								child{ node[circle,draw,scale=0.6]{} }
			}}}}}}}}}};
	\end{tikzpicture}
	
	  \caption{Left: Binary tree of minimal height. Right: Binary tree of maximal height.}
	  \label{fig:binary_tree}
	\end{center}	
\end{figure} 
For the comparison we solve equation (\ref{Ising}) for $d=16$, $\Omega =1$, 
with time step-size $h=0.01$ and two tolerance parameters $\vartheta=10^{-5}$ and $\vartheta=10^{-8}$, for the binary trees of minimal and maximal height. The ranks were limited not to exceed 200. \bcl After each time step we determine 
\begin{itemize}
\item
the maximal rank $r=\max_{\tau\le\bar\tau} r_\tau$ on the tree and 
\item the total number of entries in the basis matrices and connection tensors, which is essentially (ignoring the orthonormality constraints) the number of independent evolving parameters used to describe the system. 
\end{itemize}
Aside from the tree topology, the rank-adaptive integration algorithm used was identical for the TTNs on both trees.
\ecl
\begin{figure}[H]
	\centering
	\includegraphics[trim ={30mm 0mm 0 0mm},clip,scale=0.3]{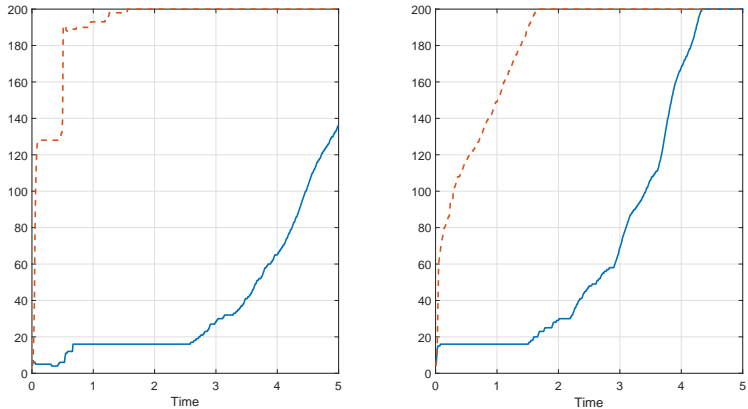}
    \caption{The figure shows the maximal ranks on the tree as selected by the rank-adaptive integrator, plotted versus time. The solid line is for the binary tree of minimal height, and the dashed line is for the binary tree of maximal height (MPS), each tree with $d=16$ leaves and $31$ vertices. Left: Maximal ranks for $\vartheta = 10^{-5}$. Right: Maximal ranks for $\vartheta = 10^{-8}$.}
    \label{fig:rank_compare}
\end{figure}
In Figure \ref{fig:rank_compare} the maximal rank is plotted after each time step. We find that the TTN on the binary tree of minimal height can approximate the solution longer with relatively small ranks compared to the TTN on the binary tree of maximal height (matrix product state -- MPS). This holds independently of the used tolerance parameter $\vartheta$.
While we observe that increasing $\vartheta$ reduces the maximal rank for the TTN on the binary tree of minimal height, this is not the case for the MPS. Although the Ising model \eqref{Ising} only has nearest-neighbor interactions that the MPS represents well, the MPS appears to need higher ranks to capture long-range effects in this model, 
compared with the TTN on the binary tree of minimal height.\footnote{\bcl As one referee commented, the observed behavior of the ranks for the two different trees corresponds to theoretical worst-case estimates for the ranks in conversions from one tensor format to the other as given in \cite[Section 12.4]{H19}.\ecl}

\bcl
A similar situation also arises when we compare the total numbers of entries in the basis matrices and connection tensors for the two trees. These numbers are plotted versus time in Figure~\ref{fig:entries_compare}. The TTN on the binary tree of minimal height needs to store (and compute with) far fewer data than the TTN on the binary tree of maximal height (MPS).
\ecl
\begin{figure}[H]
	\centering
	\includegraphics[trim ={30mm 0mm 0 0mm},clip,scale=0.3]{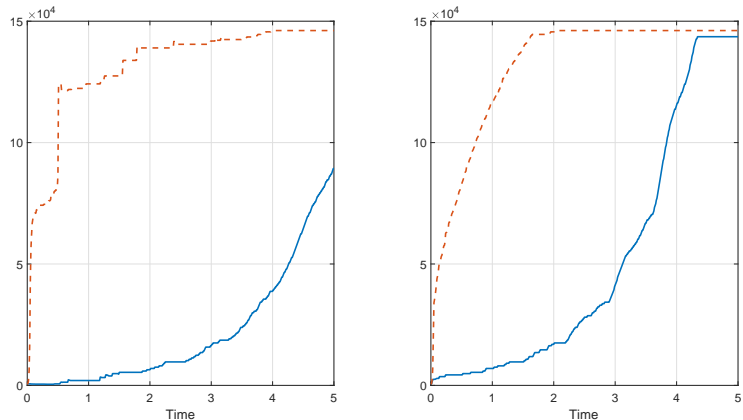}
    \caption{The figure shows the total numbers of entries in the basis matrices and connection tensors on the tree, plotted versus time. The solid line is for the binary tree of minimal height, and the dashed line is for the binary tree of maximal height (MPS), each tree with $d=16$ leaves and $31$ vertices. The near-horizontal lines on top of the figure appear because the maximal rank was limited to 200. Left: Total numbers of entries for $\vartheta = 10^{-5}$. Right: Total numbers of entries for $\vartheta = 10^{-8}$.}
    \label{fig:entries_compare}
\end{figure}


\section*{Acknowledgments} 
 We thank Federico Carollo and Igor Lesanovsky for helpful discussions on quantum spin systems and for suggesting the Ising model as a first test example. We are grateful to Federico Carollo for providing a reference solution for this model. \bcl We thank two anonymous referees for their helpful comments on a previous version. \ecl

\bcl 
This work was supported by the Deutsche Forschungsgemeinschaft (DFG, German Research Foundation) via CRC TRR 352 "Mathematics of many-body quantum systems and their collective phenomena" and FOR 5413 ``Long-range interacting quantum spin systems out of equilibrium: Experiment, theory and mathematics". The work of Gianluca Ceruti was supported by the Swiss National Science Foundation (SNSF) research project ``Fast algorithms from low-rank updates", grant number: 200020-178806. 
\ecl

\bibliography{references}{}

\begin{thebibliography}{10}

\bibitem{BaSU16}
M.~Bachmayr, R.~Schneider, and A.~Uschmajew.
\newblock Tensor networks and hierarchical tensors for the solution of
  high-dimensional partial differential equations.
\newblock {\em Found. Comput. Math.}, 16(6):1423--1472, 2016.

\bibitem{CeKL22}
G.~Ceruti, J.~Kusch, and C.~Lubich.
\newblock A rank-adaptive robust integrator for dynamical low-rank
  approximation.
\newblock {\em BIT Numer. Math.}, 2022.

\bibitem{CL21}
G.~Ceruti and C.~Lubich.
\newblock An unconventional robust integrator for dynamical low-rank
  approximation.
\newblock {\em BIT Numer. Math.}, 2021.

\bibitem{CeLW21}
G.~Ceruti, C.~Lubich, and H.~Walach.
\newblock Time integration of tree tensor networks.
\newblock {\em SIAM J. Numer. Anal.}, 59(1):289--313, 2021.

\bibitem{CiPSV21}
J.~I. Cirac, D.~Perez-Garcia, N.~Schuch, and F.~Verstraete.
\newblock Matrix product states and projected entangled pair states: Concepts,
  symmetries, theorems.
\newblock {\em Reviews of Modern Physics}, 93(4):045003, 2021.

\bibitem{LMV00}
L.~De~Lathauwer, B.~De~Moor, and J.~Vandewalle.
\newblock A multilinear singular value decomposition.
\newblock {\em SIAM J. Matrix Anal. Appl.}, 21(4):1253--1278, 2000.

\bibitem{DeRV21}
A.~Dektor, A.~Rodgers, and D.~Venturi.
\newblock Rank-adaptive tensor methods for high-dimensional nonlinear {PDEs}.
\newblock {\em J. Sci. Comput.}, 88(2):1--27, 2021.

\bibitem{DS14}
S.~V. Dolgov and D.~V. Savostyanov.
\newblock Alternating minimal energy methods for linear systems in higher
  dimensions.
\newblock {\em SIAM J. Sci. Comput.}, 36(5):A2248--A2271, 2014.

\bibitem{DuC21}
A.~J. Dunnett and A.~W. Chin.
\newblock Efficient bond-adaptive approach for finite-temperature open quantum
  dynamics using the one-site time-dependent variational principle for matrix
  product states.
\newblock {\em Phys. Rev. B}, 104(21):214302, 2021.

\bibitem{FaHN15}
A.~Falc{\'o}, W.~Hackbusch, and A.~Nouy.
\newblock Geometric structures in tensor representations (final release).
\newblock {\em arXiv preprint arXiv:1505.03027}, 2015.

\bibitem{FaHN18}
A.~Falc{\'o}, W.~Hackbusch, and A.~Nouy.
\newblock Tree-based tensor formats.
\newblock {\em SeMA J.}, pages 1--15, 2018.

\bibitem{G10}
L.~Grasedyck.
\newblock Hierarchical singular value decomposition of tensors.
\newblock {\em SIAM J. Matrix Anal. Appl.}, 31:2029--2054, 2010.

\bibitem{H19}
W.~Hackbusch.
\newblock {\em Tensor spaces and numerical tensor calculus}, volume~56 of {\em
  Springer Series in Computational Mathematics}.
\newblock Springer, Cham, 2019.
\newblock Second edition.

\bibitem{HaLOVV16}
J.~Haegeman, C.~Lubich, I.~Oseledets, B.~Vandereycken, and F.~Verstraete.
\newblock Unifying time evolution and optimization with matrix product states.
\newblock {\em Phys. Rev. B}, 94(16):165116, 2016.

\bibitem{KiLW16}
E.~Kieri, C.~Lubich, and H.~Walach.
\newblock Discretized dynamical low-rank approximation in the presence of small
  singular values.
\newblock {\em SIAM J. Numer. Anal.}, 54:1020--1038, 2016.

\bibitem{KoB09}
T.~G. Kolda and B.~W. Bader.
\newblock Tensor decompositions and applications.
\newblock {\em SIAM Review}, 51:455--500, 2009.

\bibitem{KrS81}
P.~Kramer and M.~Saraceno.
\newblock {\em Geometry of the time-dependent variational principle in quantum
  mechanics}, volume 140 of {\em Lecture Notes in Physics}.
\newblock Springer-Verlag, Berlin-New York, 1981.

\bibitem{KT14}
D.~Kressner and C.~Tobler.
\newblock Algorithm 941: {\tt htucker}--a {M}atlab toolbox for tensors in
  hierarchical {T}ucker format.
\newblock {\em ACM Trans. Math. Software}, 40(3):Art. 22, 22, 2014.

\bibitem{LiKR21}
L.~P. Lindoy, B.~Kloss, and D.~R. Reichman.
\newblock Time evolution of {ML--MCTDH} wavefunctions. {II}. {A}pplication of
  the projector splitting integrator.
\newblock {\em J. Chem. Phys.}, 155(17):174109, 2021.

\bibitem{Lu08}
C.~Lubich.
\newblock {\em From quantum to classical molecular dynamics: reduced models and
  numerical analysis}.
\newblock Zurich Lectures in Advanced Mathematics. European Mathematical
  Society (EMS), Z\"urich, 2008.

\bibitem{Lu15}
C.~Lubich.
\newblock Time integration in the multiconfiguration time-dependent {H}artree
  method of molecular quantum dynamics.
\newblock {\em Appl. Math. Res. Express}, 2015:311--328, 2015.

\bibitem{LuO14}
C.~Lubich and I.~V. Oseledets.
\newblock A projector-splitting integrator for dynamical low-rank
  approximation.
\newblock {\em BIT Numer. Math.}, 54:171--188, 2014.

\bibitem{LuOV15}
C.~Lubich, I.~V. Oseledets, and B.~Vandereycken.
\newblock Time integration of tensor trains.
\newblock {\em SIAM J. Numer. Anal.}, 53:917--941, 2015.

\bibitem{LuVW18}
C.~Lubich, B.~Vandereycken, and H.~Walach.
\newblock Time integration of rank-constrained {T}ucker tensors.
\newblock {\em SIAM J. Numer. Anal.}, 56:1273--1290, 2018.

\bibitem{Pa-etal19}
S.~Paeckel, T.~K{\"o}hler, A.~Swoboda, S.~R. Manmana, U.~Schollw{\"o}ck, and
  C.~Hubig.
\newblock Time-evolution methods for matrix-product states.
\newblock {\em Annals of Physics}, 411:167998, 2019.

\bibitem{Scho11}
U.~Schollw{\"o}ck.
\newblock The density-matrix renormalization group in the age of matrix product
  states.
\newblock {\em Annals of Physics}, 326(1):96--192, 2011.

\bibitem{ShDV06}
Y.-Y. Shi, L.-M. Duan, and G.~Vidal.
\newblock Classical simulation of quantum many-body systems with a tree tensor
  network.
\newblock {\em Phys. Rev. A}, 74(2):022320, 2006.

\bibitem{S73}
R.~B. Stinchcombe.
\newblock Ising model in a transverse field. {I}. {B}asic theory.
\newblock {\em J. Phys. C: Solid State Physics}, 6(15):2459--2483, 1973.

\bibitem{UschV13}
A.~Uschmajew and B.~Vandereycken.
\newblock The geometry of algorithms using hierarchical tensors.
\newblock {\em Linear Algebra Appl.}, 439(1):133--166, 2013.

\bibitem{Vi03}
G.~Vidal.
\newblock Efficient classical simulation of slightly entangled quantum
  computations.
\newblock {\em Phys. Rev. Letters}, 91(14):147902, 2003.

\bibitem{WaT03}
H.~Wang and M.~Thoss.
\newblock Multilayer formulation of the multiconfiguration time-dependent
  {H}artree theory.
\newblock {\em J. Chem. Phys.}, 119(3):1289--1299, 2003.

\bibitem{YaW20}
M.~Yang and S.~R. White.
\newblock Time-dependent variational principle with ancillary {K}rylov
  subspace.
\newblock {\em Phys. Rev. B}, 102(9):094315, 2020.

\end{thebibliography}


\appendix\section{\bf Rank truncation on general tree tensor networks}

\bcl
For the derivation and analysis of the rank truncation algorithm for general (not necessarily binary) complex TTNs stated in Algorithm 7 in Section~\ref{subsec:truncation},
we give a formulation in which the given TTN is first rotated (more precisely, the connection tensors are transformed using the unitary matrices that are built from the left singular vectors) and then cut (more precisely, entries of rotated connection tensors multiplying small singular values are set to zero). 

For binary TTNs, a formulation and error analysis of an HOSVD-based rank truncation algorithm was previously given 
by Hackbusch~\cite[Sections 11.3.3 and 11.4.2]{H19}. In the binary case, the truncation algorithm given here (Algorithm 7) is similar in that it computes reduced SVDs of matricizations of small tensors that have the size of the connection tensors
but it is not identical to the algorithm in~\cite{H19}. Algorithm 7 appears simpler, but our error analysis for it yields a dependence on the dimension $d$ that is linear as opposed to the square root behaviour of the truncation algorithm for binary trees in \cite{H19}. 

This appendix uses the notation of Section 2 and refers to Algorithm 7 in Section~\ref{subsec:truncation}, but is otherwise independent of the rest of the paper. The error bound given in Theorem A.1 below is used in Section~\ref{sec:cons}.

\subsection*{Derivation of a TTN rank-truncation algorithm}
We start from a rank-aug\-mented  TTN  $\wh X_{\bar\tau}$  in orthonormal representation such that for each subtree $\tau =(\tau_1,\dots\tau_m)$ of $\bar \tau$, the augmented connection tensors $\wh C_\tau \in \C^{\wh r_\tau\times \wh r_{\tau_1}\times \dots\times \wh r_{\tau_m}}$ and the basis matrices are related by
$$
\wh X_\tau = \wh C_\tau \times_0 \I_{\tau} \bigtimes_{i=1}^m \wh \U_{\tau_i}  \quad\text{ with}\quad
\wh \U_{\tau_i} = \mat_0(\wh X_{\tau_i})^\top.
$$
For $i=1,\dots,m$, we consider the reduced SVD of the $i$th matricization of $\wh C_\tau$:
$$
\mat_i(\wh C_\tau)= \wh \bfP_{\tau_i} \wh {\bf\Sigma}_{\tau_i} \wh \Q_{\tau_i}^*
$$
with the unitary matrix $\wh \bfP_{\tau_i} \in \C^{\wh r_{\tau_i}\times \wh r_{\tau_i}}$ of the left singular vectors, the diagonal matrix of singular values $\wh {\bf\Sigma}_{\tau_i}$ of the same dimension with decreasing nonnegative real diagonal entries, and further the matrix $\wh\Q_{\tau_i}$  with orthonormal columns built of the first $\wh r_\tau$ right singular vectors, which will not enter the algorithm. We set $\wh \bfP_{\bar\tau}=\I_{r_{\bar\tau}}=1$ to start the recursion, which goes from the root to the leaves.
\ecl

\medskip\noindent
(i) {\it Rotate.} We define
$$
\wh C_\tau^{\rm rot} = \wh C_\tau \times_0 \wh\bfP_{\tau}^\top \bigtimes_{i=1}^m \wh\bfP_{\tau_i}^*\quad \text{ and }\quad \wh \U_{\tau_i}^{\rm rot}  = \wh \U_{\tau_i} \wh \bfP_{\tau_i} \quad (i=1,\dots,m),
$$
for which we observe
$$
\mat_i(\wh C_\tau^{\rm rot})
= \wh \bfP_{\tau_i}^* \mat_i(\wh C_\tau) \bigotimes_{j\ne i} \bigl(\wh \bfP_{\tau_j}^*\bigr)^\top \otimes \wh \bfP_{\tau}=
\wh {\bf\Sigma}_{\tau_i} \wh \Q_{\tau_i}^* \bigotimes_{j\ne i} \bigl(\wh \bfP_{\tau_j}^*\bigr)^\top \otimes \wh \bfP_{\tau}
$$

\vspace{-5mm}
\noindent
and
$$
\wh \U_{\tau_i}^{\rm rot} = \wh \U_{\tau_i} \wh \bfP_{\tau_i} = \mat_0(\wh X_{\tau_i})^\top \wh \bfP_{\tau_i} =
\bigl( \wh \bfP_{\tau_i}^\top \mat_0(\wh X_{\tau_i}) \bigr)^\top =
\mat_0\bigl(\wh X_{\tau_i}\times_0  \wh \bfP_{\tau_i}^\top\bigr)^\top.
$$
We set 
$$
\wh X_\tau^{\rm rot} =  \wh X_\tau \times_0 \wh\bfP_{\tau}^\top = \wh C_\tau^{\rm rot} \times_0 \I_{\wh r_\tau} \bigtimes_{i=1}^m \wh \U_{\tau_i}^{\rm rot}.
$$
So we have
$$
\wh X_\tau^{\rm rot} = \wh C_\tau^{\rm rot} \times_0 \I_{\wh r_\tau} \bigtimes_{i=1}^m \wh \U_{\tau_i}^{\rm rot}
\quad\text{ with }\quad
\wh \U_{\tau_i}^{\rm rot} =\mat_0\bigl(\wh X_{\tau_i}^{\rm rot} \bigr)^\top
$$
and for $i=1,\dots,m$,  
\begin{equation}\label{mat-rot}
\mat_i(\wh C_\tau^{\rm rot}) = \wh {\bf\Sigma}_{\tau_i} \wh \V_{\tau_i}^*
\end{equation}
with a matrix $\wh \V_{\tau_i} $ having orthonormal columns. Moreover, we have $\wh X_{\bar\tau}^{\rm rot} = \wh X_{\bar\tau}$.

\medskip\noindent
(ii) {\it Cut.}
The reduced rank $r_{\tau_i} \le \wh r_{\tau_i}$ is chosen 
as the smallest integer such that  the singular values $\sigma_k$ in $\wh{\bf\Sigma}_{\tau_i}$ satisfy for a given tolerance $\vartheta>0$
\begin{equation}\label{vartheta}	
\left( \sum\limits_{k=r_{\tau_{i}}+1}^{\wh r_{\tau_i}} \sigma_k^2 \right)^{1/2} \leq \vartheta.
\end{equation}
Let ${\bf\Sigma}_{\tau_i} $ be the truncated diagonal matrix with the largest $r_{\tau_i}$ diagonal entries of $\wh {\bf\Sigma}_{\tau_i}$. For the truncated tensor $\wh C_\tau^{\rm cut}$ of the same dimension as $\wh C_\tau^{\rm rot}$ defined by
$$
\wh C_\tau^{\rm cut} \big\vert_{(k_0,k_1,\dots,k_m)} = 
\begin{cases} \wh C_\tau^{\rm rot} \big\vert_{(k_0,k_1,\dots,k_m)} & \text{if $k_0\le r_\tau$, $k_i \le r_{\tau_i}$ $(i=1,\dots,m)$}
\\ 0 & \text{else,}
\end{cases}
$$
we then have
\begin{equation}\label{mat-cut}	
\mat_i(\wh C_\tau^{\rm cut})= 
\begin{pmatrix} {\bf\Sigma}_{\tau_i} &0 \\ 0 & 0 \end{pmatrix} \wh \V_{\tau_i}^*,
\end{equation}
where the $\wh r_{\tau_i} - r_{\tau_i}$ smallest singular values of $\wh{\bf\Sigma}_{\tau_i}$ have been cut to $0$ (but nothing else is changed).
We approximate $\wh X_\tau^{\rm rot}$ by
$$
\wh X_\tau^{\rm cut} = \wh C_\tau^{\rm cut} \times_0 \I_{r_\tau} \bigtimes_{i=1}^m \wh \U_{\tau_i}^{\rm rot}.
$$
Since many entries of $\wh C_\tau^{\rm cut} $ are zero, this expression can be simplified.
We define the dimension-reduced tensor $C_\tau \in \C^{r_\tau\times r_{\tau_1}\times \dots \times r_{\tau_m}}$ as the essential part of $\wh C_\tau^{\rm cut}$, 
$$
C_\tau \big\vert_{(k_0,k_1,\dots,k_m)} = \wh C_\tau^{\rm cut} \big\vert_{(k_0,k_1,\dots,k_m)} \quad\ \text{for $k_0\le r_\tau$, $k_i \le r_{\tau_i}$ $(i=1,\dots,m)$}.
$$
Let $\bfP_{\tau} \in \C^{\wh r_{\tau}\times r_{\tau}}$ be the matrix built of the first $r_{\tau}$ columns of $\wh \bfP_{\tau}$. We then have
$$
C_\tau = \wh C_\tau \times_0 \bfP_{\tau}^\top \bigtimes_{i=1}^m \bfP_{\tau_i}^*.
$$
Let $\wt\U_{\tau_i} $ be the matrix built of the first $r_{\tau_i}$ columns of $\wh \U_{\tau_i}^{\rm rot}$, i.e.,  $\wt\U_{\tau_i}=\wh \U_{\tau_i} \bfP_{\tau_i}$. We obtain the reduced representation
$$
\wh X_{\tau}^{\rm cut}  \times_0 (\I_{r_\tau}, 0) = C_\tau \times_0 \I_{\tau} \bigtimes_{i=1}^m \wt \U_{\tau_i}.
$$
We note that
$\wt\U_{\tau_i}=\mat_0(\wt X_{\tau_i})^\top$ with
 $\wt X_{\tau_i}=\wh X_{\tau_i}\times_0 \bfP_{\tau_i}^\top=\wh X_{\tau_i}^{\rm rot}\times_0 (\I_{r_{\tau_i}},0)$, which differs from $\wh X_{\tau_i}^{\rm rot}$ with $\wh\U_{\tau_i}^{\rm rot}=\mat_0(\wh X_{\tau_i}^{\rm rot})^\top$ only in that the connection tensor
$\wh C_{\tau_i}^{\rm rot}$ is replaced by the smaller tensor $\wt C_{\tau_{i}} =\wh C_{\tau_i}^{\rm rot} \times_0 (\I_{r_{\tau_i}},0)$.

(iii) {\it Recursion.} The rotate-and-cut procedure is done recursively from the root to the leaves, where finally we set $\U_l=\wt \U_l=\wh \U_l \bfP_l$. In this way we obtain the rank-truncated TTN $X_{\bar\tau}$ with the reduced connection tensor $C_\tau$ for each subtree $\tau=(\tau_1,\dots,\tau_m)$ of $\bar\tau$ and the reduced basis matrix $\U_l$ for each leaf $l$:
$$
X_\tau = C_\tau \times_0 \I_{\tau} \bigtimes_{i=1}^m \U_{\tau_i}
\quad\text{ with}\quad
\U_{\tau_i} = \mat_0(X_{\tau_i})^\top.
$$
We note that this representation of $X_\tau$ is in general not orthonormal, but as we show below, the norms of the factors behave in a stable way. If orthonormality is needed for output, the TTN can be reorthonormalized as described after Lemma~\ref{lem:orth}. Orthonormality is, however, not needed for advancing with the next time step, as the integrator anyway computes new orthonormal bases in Algorithm~\ref{alg:Phi-tau-i}.

With the above considerations we arrive at Algorithm~\ref{alg:truncation-tau}, which arranges the computations in a different but mathematically equivalent way. 

\subsection*{Error analysis of the rank-truncation algorithm}
We are given the TTN $\wh X_{\bar\tau}$ in orthonormal representation, which is built up recursively,
for each subtree $\tau=(\tau_1,\dots,\tau_m)\le \bar\tau$, from the orthonormal sub-TTN
$$
\wh X_\tau = \wh C_\tau \times_0 \I_{\wh r_\tau} \bigtimes_{i=1}^m \wh \U_{\tau_i}
\quad\text{ with } \quad \wh \U_{\tau_i} = \mat_0(\wh X_{\tau_i})^\top,
$$
where $\wh C_\tau  \in \C^{\wh r_\tau\times \wh r_{\tau_1}\times \dots\times \wh r_{\tau_m}}$.
For $\tau\ne\bar\tau$ the columns of
$\mat_0(\wh C_\tau)$ and $\wh \U_{\tau}$ 
are orthonormal,
so that in particular their matrix 2-norms are equal to 1:
\begin{equation} \label{hat-CU-norm}
\| \mat_0(\wh C_\tau) \|_2 = 1, \quad \| \wh \U_{\tau} \|_2 = 1
\qquad\text{ for }\ \tau < \bar\tau.
\end{equation}
Algorithm~\ref{alg:truncation-tau} computes a rank-truncated TTN $X_{\tau}$ 
for $\tau\le\bar\tau$,
$$
 X_\tau = C_\tau \times_0 \I_{r_\tau} \bigtimes_{i=1}^m \U_{\tau_i} 
\quad\text{ with } \quad  \U_{\tau_i}  = \mat_0(X_{\tau_i})^\top,
$$
where $C_\tau \in \C^{r_\tau\times r_{\tau_1}\times \dots \times r_{\tau_m}}$. The reduced ranks $r_\tau \le \wh r_\tau$ depend on the given tolerance parameter $\vartheta$ as used in \eqref{vartheta}.	

\bcl
In the following error bound, the norm $\| X \|$ of a tensor $X$ is the Euclidean norm of the vector of entries of $X$, and the integer $d_\tau$ is the number of vertices of $\tau$.

\begin{theorem}[Rank truncation error]
\label{thm:err-trunc}
The error of the tree tensor network~$X_{\bar\tau}$, which results from rank truncation of $\wh X_{\bar\tau}$ with tolerance $\vartheta$ according to Algorithm~7, is bounded by 
$$
\| X_{\bar\tau}-\wh X_{\bar\tau} \| \le c_{\bar\tau}\, \vartheta
\ \ \text{ with }\ 
c_{\bar\tau} =\| C_{\bar\tau} \| (d_{\bar\tau}-1) +1.
$$
\end{theorem}

{\sc Remark.} {\it If the tolerance in the truncation at the root, from $\wh C_{\bar\tau}$ to $ C_{\bar\tau}$, is modified to 
$\vartheta / \| \wh C_{\bar\tau} \|$ but is left at $\vartheta$ at the other vertices, then the proof yields the error bound}
$$
\| X_{\bar\tau}-\wh X_{\bar\tau} \| \le d_{\bar\tau}\, \vartheta.
$$
\ecl

\begin{proof} 
From the sub-TTNs $\wh X_{\tau}$ with $\tau\le \bar\tau$ we construct the rotated and cut TTNs $\wh X_{\tau}^{\rm rot}$, $\wh X_{\tau}^{\rm cut}$ as described above. For a tree $\tau=(\tau_1,\dots,\tau_m)\le\bar\tau$ we then have the corresponding matrices
\begin{align*}
\wh \U_\tau^{\rm rot} &= \mat_0(\wh X_{\tau}^{\rm rot})^\top, \qquad 
\wh X_{\tau}^{\rm rot} = \wh C_{\tau}^{\rm rot} \times_0 \I_{\wh r_\tau} \bigtimes_{i=1}^m \wh \U_{\tau_i}^{\rm rot},
\\
 \U_\tau &= \mat_0(X_{\tau})^\top, \qquad  \quad\
X_\tau = C_{\tau}  \times_0 \I_{r_\tau} \bigtimes_{i=1}^m \U_{\tau_i}.
\end{align*}
With
$$
\wt \U_\tau = \wh \U_\tau^{\rm rot}\begin{pmatrix} \I_{r_\tau} \\ 0 \end{pmatrix} =  \mat_0(\wt X_{\tau})^\top,
\qquad \wt X_{\tau}= \wh X_{\tau}^{\rm rot} \times_0 (\I_{r_\tau}, 0),
$$ 
we further have the intermediate tensor
\begin{align*}
\wt X_{\tau}^{\rm cut} =
\wh X_{\tau}^{\rm cut}  \times_0 (\I_{r_\tau}, 0) &= \,\wh C_{\tau}^{\rm cut} \times_0 (\I_{r_\tau}, 0) \bigtimes_{i=1}^m \wh \U_{\tau_i}^{\rm rot}
\\
&= \ \ \; C_{\tau}  \times_0 \I_{r_\tau} \bigtimes_{i=1}^m \wt\U_{\tau_i}.
\end{align*}
We prove the error bound 
\begin{equation}\label{U-err}
\| \U_\tau-\wt \U_\tau \|_2\le d_\tau \,\vartheta\quad\text{ for }\ \tau<\bar\tau
\end{equation}
using induction over the height of the tree. At leaves~$l$ we have $\U_l=\wt \U_l$ and $\| \wt \U_l\|_2 \le 1$ by construction. We make the induction hypothesis 
$$
\text{$\| \U_{\tau_i} \|_2 \le 1$ and $\| \U_{\tau_i} - \wt \U_{\tau_i} \|_2 \le d_{\tau_i} \, \vartheta$.}
$$
We note that $\|  \wt \U_{\tau_i} \|_2 \le 1$ by \eqref{hat-CU-norm} and by the construction of  $\wt \U_{\tau_i}$ from $\wh \U_{\tau_i}$.
We further observe that
\begin{align*}
&\| \mat_0(C_\tau) \|_2 = \| \mat_0(\wh C_{\tau}^{\rm cut}) \|_2 \le \| \mat_0(\wh C_{\tau}^{\rm rot}) \|_2 = \| \mat_0(\wh C_{\tau}) \|_2
\quad\text{ and }
\\
&\| \U_\tau \|_2 = \| \mat_0(X_\tau) \|_2 = \|  \mat_0(C_\tau)  \bigotimes_{i=1}^m \U_{\tau_i}^\top \|_2 \le \| \mat_0(C_\tau) \|_2
\prod_{i=1}^m \| \U_{\tau_i} \|_2.
\end{align*}
By \eqref{hat-CU-norm} and the induction hypothesis, we thus obtain
\begin{equation}\label{CU-norm}
\| \mat_0(C_\tau) \|_2 \le 1 \quad \text{ and }\quad \| \U_\tau \|_2 \le 1.
\end{equation}
We have $\| \U_\tau-\wt \U_\tau \|_2 = \| \mat_0 ( X_\tau - \wt X_\tau ) \|_2$ and write
$$
 X_\tau - \wt X_\tau = \bigl( X_\tau - \wt X_{\tau}^{\rm cut}  \bigr) +
\bigl( \wh X_{\tau}^{\rm cut} - \wh X_{\tau}^{\rm rot} \bigr) \times_0 (\I_{r_\tau}, 0).
$$
The first term on the right-hand side equals
$$
X_\tau - \wt X_{\tau}^{\rm cut}  = C_{\tau}  \times_0 \I_{r_\tau} \bigtimes_{i=1}^m \U_{\tau_i}-C_{\tau}  \times_0 \I_{r_\tau} \bigtimes_{i=1}^m \wt \U_{\tau_i}
$$
with the same connection tensor $C_\tau$ in both terms of the difference. We then have
$$
\mat_0\bigl( X_\tau - \wt X_{\tau}^{\rm cut} \bigr) = \mat_0(C_\tau) \Bigl( \bigotimes_{i=1}^m \U_{\tau_i}^\top - \bigotimes_{i=1}^m \wt \U_{\tau_i}^\top \Bigr).
$$
Writing the difference of the Kronecker products as a telescoping sum, using that  $\mat_0(C_\tau)$ and $\U_{\tau_i}$ are bounded by $1$ in the matrix 2-norm, and finally using the induction hypothesis, we obtain
$$
\| \mat_0\bigl( X_\tau - \wt X_{\tau}^{\rm cut} \bigr) \|_2 \le \sum_{i=1}^m \| \U_{\tau_i} - \wt \U_{\tau_i} \|_2
\le \sum_{i=1}^m d_{\tau_i} \, \vartheta = (d_\tau-1)\vartheta.
$$
On the other hand,
\begin{align*}
\bigl( \wh X_{\tau}^{\rm cut} - \wh X_{\tau}^{\rm rot} \bigr) \times_0 (\I_{r_\tau}, 0) = \bigl(  \wh C_{\tau}^{\rm cut} - \wh C_{\tau}^{\rm rot} \bigr) \times_0 (\I_{r_\tau}, 0)  \bigtimes_{i=1}^m \wh \U_{\tau_i}^{\rm rot},
\end{align*}
where $\| \wh \U_{\tau_i}^{\rm rot} \|_2 \le 1$ by construction and \eqref{hat-CU-norm}. So we have
\begin{align*}
&\| \mat_0 \bigl( (\wh X_{\tau}^{\rm cut} - \wh X_{\tau}^{\rm rot} ) \times_0 (\I_{r_\tau}, 0) \bigr) \|_2 \le \| \mat_0 \bigl( \wh C_{\tau}^{\rm cut} - \wh C_{\tau}^{\rm rot} \bigr) \|_2 
\\
&\le \| \mat_0 \bigl( \wh C_{\tau}^{\rm cut}\bigr) - \mat_0 \bigl( \wh C_{\tau}^{\rm rot} \bigr) \|_F  = \| \mat_1 \bigl( \wh C_{\tau}^{\rm cut}\bigr) - \mat_1 \bigl( \wh C_{\tau}^{\rm rot} \bigr) \|_F\le \vartheta,
\end{align*}
where we used \eqref{mat-rot}--\eqref{mat-cut} in the last inequality. Altogether we find
\begin{align*}
\| \U_{\tau} - \wt \U_{\tau} \|_2 &\le \| \mat_0\bigl( X_\tau - \wt X_{\tau}^{\rm cut} \bigr) \|_2  + \| \mat_0 \bigl( (\wh X_{\tau}^{\rm cut} - \wh X_{\tau}^{\rm rot} ) \times_0 (\I_{r_\tau}, 0) \bigr) \|_2 
\\
&\le d_\tau \, \vartheta,
\end{align*}
which completes the proof of \eqref{U-err} by induction. Finally, for the full tree $\bar\tau=(\tau_1,\dots,\tau_m)$, where $\wh r_{\bar\tau}=r_{\bar\tau}=1$ but $\| C_{\bar\tau} \|$ is arbitrary, we use the same argument as above in estimating the norm of
$$
X_{\bar\tau} - \wh X_{\bar\tau} = X_{\bar\tau} - \wh X_{\bar\tau}^{\rm rot} =
\bigl( C_{\bar\tau} \bigtimes_{i=1}^m \U_{\tau_i} - C_{\bar\tau} \bigtimes_{i=1}^m \wt \U_{\tau_i} \bigr) +
\bigl(  \wh C_{\bar\tau}^{\rm cut} - \wh C_{\bar\tau}^{\rm rot} \bigr) \bigtimes_{i=1}^m \wh \U_{\tau_i}^{\rm rot}.
$$
This yields
$
\| X_{\bar\tau} - \wh X_{\bar\tau} \| \le \bigl( \| C_{\bar\tau} \| \, (d_{\bar\tau}-1)  + 1 \bigr)\vartheta. 
$
\end{proof}

\end{document}